\documentclass[reqno]{amsart}
\usepackage{header}
\begin{document}

\title{Distality Rank}

\author{Roland Walker}
\address{University of Illinois at Chicago\\Department of Mathematics, Statistics, and Computer Science\\851 S Morgan St\\322 SEO, MC 249\\Chicago, Illinois 60607-7045}
\email{rwalke20@uic.edu}
\begin{abstract}
Building on Pierre Simon's notion of distality, we introduce distality rank as a property of first-order theories and give examples for each rank $m$ such that $1\leq m \leq \omega$. For NIP theories, we show that distality rank is invariant under base change. We also define a generalization of type orthogonality called $m$-determinacy and show that theories of distality rank $m$ require certain products to be $m$-determined. Furthermore, for NIP theories, this behavior characterizes $m$-distality. If we narrow the scope to stable theories, we observe that $m$-distality can be characterized by the maximum cycle size found in the forking ``geometry,'' so it coincides with $(m-1)$-triviality. On a broader scale, we see that $m$-distality is a strengthening of Saharon Shelah's notion of $m$-dependence.
\end{abstract}


\maketitle

\section{Introduction}

In this paper, we define and develop distality rank and strong distality rank as classification tools for first-order model theory. The local versions of these ranks classify EM-types; however, in the natural fashion, they may also be applied globally to classify first-order theories. Both ranks are generalizations of distality which was introduced in 2013 by Pierre Simon \cite{simon:distalandnondistal}. 

The introduction of distality was motivated as an attempt to better understand unstable NIP theories by studying their stable and ``purely unstable,'' or distal, parts separately. This decomposition is particularly easy to see for algebraically closed valued fields where the stable part is the residue field and the distal part is the value group. The approach of studying stable and distal parts independently can also be applied to types over NIP theories where each type can be decomposed into a generically stable partial type and an order-like quotient \cite{simon:typedecomposition}.

Distality quickly became interesting and useful in its own right, and much progress has been made in recent years studying distal theories. Such a theory exhibits no stable behavior since it is dominated by its order-like component. There are many interesting examples:
All o-minimal theories are distal, and so are the $p$-adics \cite{simon:distalandnondistal}. Philipp Hieronymi and Travis Nell developed criteria for determining when certain expansions of o-minimal theories remain distal \cite{hieronymi:distalpairs}, and Nell continued this work by studying distal behavior in dense pairs of o-minimal structures \cite{nell:distalbehavior}. In 2018, the asymptotic couple of the field of logarithmic transseries was shown to be distal by Allen Gehret and Elliot Kaplan \cite{gehret:distalitytransseries}, and in 2020, Aschenbrenner, Chernikov, Gehret, and Ziegler explored distality in valued fields and, among other things, proved that the differential field of the logarithmic-exponential transseries is distal \cite{aschenbrenner:distalityinvaluedfields}.

Many classical combinatorial results can be improved when study is restricted to objects definable in distal structures. Moreover, in \cite{chernikov:cuttinglemma}, where they developed a definable version of the Cutting Lemma, Artem Chernikov, David Galvin, and Sergei Starchenko proposed that ``distal structures provide the most general natural setting for investigating questions in `generalized incidence combinatorics.'\,'' In \cite{boxall:definablepq}, Gareth Boxall and Charlotte Kestner proved that a definable version of the $(p,q)$-Theorem, first conjectured by Chernikov and Simon in \cite{chernikov:externallydefinableII}, holds for distal structures.

Perhaps the most notable combinatorial result was obtained by Chernikov and Starchenko. In \cite{chernikov:regularitylemma}, they presented a definable version of the Szemer\'{e}di Regularity Lemma for distal structures. Although their result applies to infinite, as well as finite, $k$-partite hypergraphs, for easier comparison to the standard Szemer\'{e}di Regularity Lemma, we state their findings for finite graphs: Given $\MM$ a distal structure and $E\subseteq M\times M$ a definable edge (i.e., symmetric and irreflexive) relation, there is a constant $c$ such that for all finite induced graphs $(V, E)$ and all $\varepsilon >0$, there is a uniformly definable partition $P$ of $V$ with size $O(\varepsilon^{-c})$ whose defect $D\subseteq P\times P$ is bounded by 
\[
\sum_{(A,B)\in D}\abs{A}\abs{B} \leq \varepsilon \abs{V}^2
\]
such that the induced bipartite graph $(A,B,E)$ on every non-defective pair $(A,B)\in (P\times P)\setminus D$ is homogeneous (i.e., complete or empty).

In the same paper \cite{chernikov:regularitylemma}, Chernikov and Starchenko developed a definable version of the strong Erd\H{o}s-Hajnal property and showed that this property fully characterizes distal structures. Many other interesting characterizations of distality exist. For example, Itay Kaplan, Saharon Shelah, and Pierre Simon showed that an NIP theory has exact saturation if and only if it is not distal \cite{kaplan:exactsaturation}. 

\bigskip

Distal theories can be characterized by the following property: if 
\[
	\II_0 +\II_1 +\II_2+\cdots +\II_{n-1}+\II_n
\] 
is an indiscernible sequence, where each cut is Dedekind (i.e., the cut has no immediate predecessor or successor), and $A = (a_0,\ldots,a_{n-1})$ is such that each sequence
\begin{align*}
\II_0 + a_0 + \II_1 +\II_2 &+\cdots +\II_{n-1} + \II_n, \\
\II_0 + \II_1 +a_1 + \II_2 &+\cdots +\II_{n-1} + \II_n, \\
&\vdots \\
\II_0 + \II_1 +\II_2 +\cdots &+\II_{n-1} + a_{n-1}+ \II_n
\end{align*}
is indiscernible, then the sequence
 \[
 	\II_0 + a_0 +\II_1+a_1 +\II_2+a_2 +\cdots +\II_{n-1} + a_{n-1}+ \II_n
 \] 
 is also indiscernible. In other words, if we check that $\II$ remains indiscernible after inserting each singleton of $A$ by itself, then $\II$ remains indiscernible after inserting all of $A$ simultaneously. It seems natural to study weaker forms of this property.
 Our research program was motivated by the following questions:
 
 \begin{ques} \label{ques:dr2}
 	Are there theories where it is not always sufficient to check the singletons of $A$, but it is always sufficient to check the pairs of $A$? 
 \end{ques}
\begin{ques}\label{ques:drm}
	 Are there theories where it is not always sufficient to check the elements of $[A]^{m-1}$, but it is always sufficient to check the elements of $[A]^{m}$? 
\end{ques}
 \begin{ques}
 
    Is it interesting to study the generalizations of distality suggested by Questions \ref{ques:dr2} and \ref{ques:drm} outside NIP?
 \end{ques}

This paper answers these questions in the affirmative and introduces the notion of distality rank. 
In Section \ref{s:distality rank}, we develop the concept of distality rank for EM-types and theories in such a way that a theory which satisfies the condition of Question \ref{ques:drm} is said to have \emph{distality rank $m$}.
In particular, a theory has distality rank 1 if and only if it is distal (see \cite[Definition 3.1]{simon:distalandnondistal}).
Distality rank is robust in many ways. For example, adding named parameters to a theory does not increase its distality rank (Proposition \ref{prop:DRA}); furthermore, if the theory is NIP, its distality rank is completely unaltered by base changes (Theorem \ref{thm:basechange}).

In Section \ref{s:strong distality}, we define strong distality rank which generalizes Simon's ``external characterization of distality'' (see \cite[Lemma 2.7]{simon:distalandnondistal}). Proposition \ref{prop:boundsdr} puts a bound on strong distality rank, and thus distality rank, for theories with quantifier elimination in languages where every function symbol is unary. In Section \ref{s:examples}, we use this result to give examples of theories for each distality rank. It is interesting to note that, although distality rank 1 completely excludes stable theories (Proposition \ref{prop:stablenotdistal}), we find several stable theories with distality rank 2. 
Thus, higher distality ranks no longer isolate ``purely unstable'' behavior but, rather, measure the degree to which products of certain invariant types behave deterministically as discussed later in Section \ref{s:type determinacy}.

In Section \ref{s:m-dependence}, we show that $m$-distality is a strengthening of Shelah's notion of $m$-dependence. 
For a nice overview of $m$-dependence, see \cite{chernikov:n-dependence}. 
After its introduction in \cite{shelah:definablegroupsfordependent} and \cite{shelah:stronglydependenttheories}, there has been substantial research into the structural consequences of $m$-dependence. 
For example, in a series of papers, Artem Chernikov and Nadja Hempel have been exploring the properties of $m$-dependent groups and fields. 
In  \cite{hempel:ndepgroupfields}, Hempel shows that $m$-dependent fields are Artin-Schreier closed, and in \cite{chernikov:n-depgroupfieldsII}, Chernikov and Hempel work towards proving a conjecture that there are no strictly $m$-dependent fields for $m\geq 2$, garnering several interesting results along the way. They also show that $m$-dependence is preserved by Mekler's construction \cite{chernikov:Mekler}. In  \cite{terry:vcldimension}, Caroline Terry uses $m$-dependence and the associated $\VC_m$-dimension, or rather its dual, to improve what was previously known about the jumps in the speeds of hereditary $L$-properties. 
Finally, Artem Chernikov and Henry Towsner develop a hypergraph regularity lemma based on $\VC_m$-dimension \cite{chernikov:HypergraphRegularity}. 
It is reasonable to conjecture that $m$-distality will further strengthen some of these structural results in much the same way that requiring distality strengthened several results previously known for NIP. 
In particular, a long-term goal of our research program is to produce a more homogeneous version of the Chernikov--Towsner hypergraph regularity lemma \cite[Theorem 1.1]{chernikov:HypergraphRegularity} using $m$-distality.  
 
Several of the structural consequences of distality have analogues for theories of higher distality rank.
For example, in \cite{simon:distalandnondistal}, Simon proves that an NIP theory is distal if and only if any two global invariant types which commute are orthogonal. Recall that two global invariant types $p(x)$ and $q(y)$ are \emph{orthogonal} exactly when their union $p(x)\cup q(y)$ completely determines their product $(p\otimes q)(x,y)$. We introduce the notion of $m$-determinacy which generalizes orthogonality (see Definition \ref{def:mdetermined}). In particular, the types $p$ and $q$, as above, are orthogonal if and only if their product $p\otimes q$ is 1-determined. In $m$-distal theories (NIP or IP), every product $p_0\otimes\cdots\otimes p_{n-1}$ of global invariant types which commute pairwise is $m$-determined (Proposition \ref{prop:distalortho}). Furthermore, if the theory is NIP, this property characterizes $m$-distality (Theorem \ref{thm:distalortho}).

Finally, in Section \ref{s:geometric stability}, we explore distality rank in the context of stable theories and prove that $m$-distality coincides with $(m-1)$-triviality as defined by Goode in Section 1 of \cite{goode:trivialconsiderations}. 
(See Theorem \ref{thm:distal iff trivial}.)
If a superstable theory is not trivial, then we can find a 3-cycle (see Definition \ref{def:cycle}) among the realizations of some (possibly imaginary) regular type. 
Furthermore, since forking dependence defines a pregeometry on the set of realizations of a regular type, we can ``expand'' that 3-cycle to form arbitrarily large cycles. 
Thus, if a superstable theory is trivial, its distality rank is $2$. 
Otherwise, its distality rank is $\omega$. (See Proposition \ref{prop:DR gap for superstable}.)

\subsection*{Acknowledgments}
The author would like to thank John Baldwin, James Freitag, Anand Pillay, and the anonymous referee for helpful comments on earlier drafts, Pierre Simon for suggesting that $\ORPG_m$ provides examples of $\EM$-types whose distality rank and strong distality rank disagree, Chris Laskowski for helpful discussions concerning the results appearing in Section \ref{s:geometric stability}, Artem Chernikov for allowing the inclusion of Proposition \ref{prop:m-distal implies m-dependent}, and David Marker for continued support and thoughtful advice throughout the entire research and writing process.

\section{Preliminaries and Notation}\label{s:preliminaries}

Throughout this paper, unless otherwise specified, assume we have fixed $L$ an arbitrary language, $T$ a complete first-order $L$-theory with infinite models, and $\mathcal{U} \models T$ a monster model which is universal and strongly $\hat{\kappa}$-homogeneous for some sufficiently large cardinal $\hat{\kappa}$ (see\ \cite[Definition 6.15 and Theorem 6.16]{tent:modeltheory}). We say a set is \emph{small} if its cardinality is strictly less than $\hat{\kappa}$; otherwise, we say the set is \emph{large}.

If $A \subseteq U$ is a set of parameters, we use $L_A$ to denote the language $L\cup \left\{ c_a:a\in A\right\}$ where each $c_a$ is a constant symbol, $\mathcal{U}_A$ to denote the expansion of $\mathcal{U}$ to the $L_A$-structure satisfying $c_a=a$ for each $a\in A$, and $T_A$ to denote $\Th\left(\mathcal{U}_A\right)$ the full theory of the expansion.

We frequently overload a language symbol $L$ using it to denote the set of all $L$-formulae. If we wish to specify free variables, we use $L(x_0, \ldots , x_{n-1})$ to denote the set of all $L$-formulae with free variables among $x_0, \ldots , x_{n-1}$. Alternatively, we may write $L(\kappa)$ for $L(x)$ where $\kappa = \abs{x}$ is the tuple size of the free variable. 

\begin{note}
	Any variable or parameter may be a tuple, finite or infinite, unless otherwise specified.
\end{note}

\subsection{Types and Type Spaces}\label{ss:types and type spaces}

Let $b\in U$. We use $\tp_A(b)$ to denote the complete type of $b$ over $A$; i.e.,
\[
	\tp_A(b) = \{ \phi \in  L_A(|b|) \, : \, \mathcal{U} \models \phi(b) \}.	
\]
  Given $b_1, b_2 \in U$, we write $b_1 \equiv_A b_2$ if $b_1$ and $b_2$ have the same complete type over $A$. We use $S_A(x)$ to denote the set of all $x$-types over $A$; i.e., 
\[
	S_A(x) = \left\{ \tp_A(b) \, : \, b \in U^{\abs{x}} \right\}.
\]
Alternatively, we may write $S_A(\kappa)$ where $\kappa = \abs{x}$ is the tuple size of the free variable. Of course, if we omit the subscript $A$ in any of the above, we mean to be working in $L$ and $T$ without named parameters.

Suppose $A\subseteq B \subseteq D \subseteq U$ and $p\in S_B(x)$. We use $p\dhr_A$ to denote the restriction of $p$ to a type in $S_A(x)$.
Furthermore, if $p$ is invariant over $A$ (see \mbox{Definition \ref{def:invariant} and following}) and there is a unique type in $S_D(x)$ which extends $p$ while remaining invariant over $A$, we use $p\uhr^D$ to denote that extension.

\subsection{Indiscernible Sequences and EM-Types}
Let $(x_k \, : \, k < \omega)$ be a sequence of variables of uniform tuple size, finite or infinite.  Suppose $A \subseteq U$ is a small set of parameters. 

\begin{definition}
A \emph{partial EM-type} over $A$ with tuple size $\abs{x_0}$ is any
\[\Gamma \subseteq  L_A(x_k \, : \, k < \omega)\]
which is consistent with the collection of all formulae
\[
	\phi(x_0, \ldots, x_{n-1}) \leftrightarrow \phi(x_{k_0}, \ldots, x_{k_{n-1}})
\]
where $n < \omega$, $\phi \in  L_A(x_0, \ldots, x_{n-1})$, and $k_0 < \cdots < k_{n-1} < \omega$. We denote the set of all complete EM-types over $A$ with tuple size $\abs{x_0}$ as 
\[
	S_A^{\EM}(x_k \, : \, k < \omega).
\]
Alternatively, we may write $S_A^{\EM}(\kappa \cdot \omega)$ where $\kappa = \abs{x_0}$.
\end{definition}

Let $\II = (b_i \, : \, i \in I) \subseteq U^{\abs{x_0}}$ be a sequence indexed by some infinite linear order $(I,<)$.

\begin{note}
Throughout this paper, all sequences of parameters are assumed to be small.	
\end{note}

\begin{definition}
Given $\phi \in  L_A(x_0, \ldots, x_{n-1})$ for some $n < \omega$, we write $\II \EMM \phi$ if
\[
	\mathcal{U} \models \phi(b_{i_0}, \ldots, b_{i_{n-1}})
\]
for all $i_0 < \cdots <  i_{n-1} \in I$.
\end{definition}

\begin{definition}
 The \emph{partial EM-type of $\II$ over $A$} is defined as follows:
\[
	\EM_A(\II) = \{ \phi \in  L_A(x_k \, : \, k < \omega) \, : \, \II \models^{\EM} \phi \}.
\]
If $\JJ$ is also an infinite sequence and $\EM_A(\II) = \EM_A(\JJ)$, we write $\II \equiv_A^{\EM} \JJ$. If $\EM_A(\II)$ is complete, then we say that $\II$ is \emph{indiscernible} over $A$. In this case, we will often use the notation $\tp_A^{\EM}(\II)$ to emphasize that the EM-type is complete. 
\end{definition}

\begin{definition}
	We say a collection $(\II_\alpha : \alpha < \lambda)$ of infinite sequences is \emph{mutually indiscernible over} $A$ if each $\II_\alpha$ is indiscernible over $A \cup \bigcup_{\beta \ne \alpha} \II_\beta$.
\end{definition}

\subsection{Alternation Rank and NIP}

\begin{definition}
If $\phi \in  L_U(x)$ and $\II=(b_i:i\in I)\subseteq U ^{\abs{x}}$ is an infinite indiscernible sequence indexed by $(I,<)$, we use $\alt(\phi, \II)$ to denote the \emph{number of alternations of $\phi$ on $\II$}, i.e., 
\[ 
	\alt(\phi, \II)=\sup\left\{n<\omega \; : \; \exists \, i_0<\cdots < i_n\in I \quad \UU \models \bigwedge_{j<n} \neg [\phi(b_{i_j}) \leftrightarrow \phi(b_{i_{j+1}})] \right\} .
\]
Furthermore, we use $\alt(\phi)$ to denote the \emph{alternation rank of $\phi$}, i.e.,
\[ \alt(\phi)= \sup\left\{ \alt(\phi, \II) \; : \; \II\subseteq U^{\abs{x}} \text{ is an infinite indiscernible sequence} \right\} . \]
\end{definition}

\begin{definition}
A formula $\phi \in  L(x,y)$ is \emph{IP} if there is a $d \in U^{\abs{y}}$ such that $\alt(\phi(x,d))=\infty$. Moreover, the theory $T$ is \emph{IP} if there is a $\phi\in  L_U(x)$ with $\alt(\phi)=\infty$.
In both cases, we use \emph{NIP} to denote the, often more desirable, condition of not being IP. 
\end{definition}

\subsection{Cuts and Partitions}

Let  $(I, <)$ be an infinite linear order.
%
%

\begin{definition}
	We call an ordered pair $\mathfrak{c} = \left( A,B \right)$ of nonempty subsets of $I$ a \emph{cut} of $I$, and write $I=A+B$, if
	\begin{itemize}
		\item $I=A \cup B$ and 
		\item $A<B$ (i.e., $\forall a \in A\; \forall b\in B\; a<b$).
	\end{itemize}
	We often denote the left side of the cut as $\mathfrak{c}^-$ and the right side of the cut as $\mathfrak{c}^+$ (i.e., $\mathfrak{c}^- = A$ and $\mathfrak{c}^+ = B$).
\end{definition}

\begin{definition}
	A cut $\mathfrak{c}$ is \emph{Dedekind} if
	\begin{itemize}
		\item $\cc^-$ has no maximum element and
		\item $\cc^+$ has no minimum element.
	\end{itemize}
\end{definition}

\begin{definition}
	If $A$ and $B$ are nonempty subsets of $I$ such that
	\begin{itemize}
		\item $A<B$ and
		\item no element of $I$ separates $A$ from $B$ (i.e., $\nexists i\in I \; A<i<B$),
	\end{itemize}
	we write $\cut(A,B)$ to denote the unique cut of the form $I=A'+B'$ with $A\subseteq A'$ and $B\subseteq B'$.
\end{definition}

The notation $I=A+B$ indicates that the cut determines a partition of $I$.

\begin{definition}
	If $I=I_0 \cup \cdots \cup I_n$ and $I_0 < \cdots < I_n$, we write $I_0 + \cdots + I_n$ to denote the partition of $I$ determined by the cuts $\mathfrak{c}_i = \cut\left(I_i, I_{i+1}\right)$.
	Moreover, we call that partition \emph{Dedekind} if each of the cuts $\mathfrak{c}_i$ is Dedekind.
\end{definition}
\noindent When discussing a partition $I_0+\cdots +I_n$, we often assume the cuts are labeled as above, where $\cc_i=\cut(I_i,I_{i+1})$, unless otherwise specified.

\subsection{Limit Types}

Let $(I, <)$ be a linear order, and let $\II = (b_i \, : \, i \in I) \subseteq U$ be a sequence of tuples.

\begin{definition}
	Given $A\subseteq U$, if the partial type
	\[
	\{ \phi \in  L_A(x) \, : \, \exists i \in I \; \forall j \geq i \; \; \mathcal{U} \models \phi(b_j) \}
	\]
	is complete, we call it the \emph{limit type of $\II$ over $A$}, written $\limtp_A(\II)$. Moreover, if it exists, we call $\limtp_U(\II)$ the \emph{global limit type of $\II$} and may simply write $\lim(\II)$.
\end{definition}

\noindent Notice that if $\II$ is indiscernible, then $\limtp_\II(\II)$ exists. Furthermore, since NIP formulae have finite alternation rank, when $T$ is NIP and $\II$ is indiscernible, the global limit type $\lim(\II)$ exists.

\begin{definition} 
Given $A\subseteq U$ and cuts $\mathfrak{c}_0, \ldots , \mathfrak{c}_{n-1}$ of $I$, we define the limit type 
$\limtp_{A}(\mathfrak{c}^{\bullet}_0, \ldots , \mathfrak{c}^{\bullet}_{n-1})$, where each $\cc^{\bullet}_i \in \{ \cc_i, \, \cc^-_i, \cc^+_i \}$, as follows: given $\phi\in  L_A(x_0,\ldots, x_{n-1})$,
	
\begin{align*}
	\phi\in \limtp_{A}(\mathfrak{c}^{\bullet}_0, \ldots , \mathfrak{c}^{\bullet}_{n-1}) \text{ iff} &\text{ there exists }(j_0,\ldots, j_{n-1})\in \mathfrak{c}^-_0\times \cdots \times \mathfrak{c}^-_{n-1} \\ 
	& \phantom{\text{ there  i }} \text{ and }(k_0,\ldots, k_{n-1})\in \mathfrak{c}^+_0\times \cdots \times \mathfrak{c}^+_{n-1}\\ 
	&\text{ such that } \mathcal{U}\models \phi(b_{i_0}, \ldots, b_{i_{n-1}}) \\
	&\phantom{\text{ there  i }}  
	\text{ for all } (i_0, \ldots, i_{n-1}) \in \prod_{i<n} (j_i, k_i)\cap C_i \\
	&\phantom{\text{ there  i }} 
	\text{where each } C_i= \left\{\begin{array}{ll}
		I & \text{if } \cc^\bullet_i \text{ is } \cc_i , \\
		\cc^\bullet_i & \text{otherwise.} 
	\end{array}\right.
\end{align*}
Moreover, if it exists, we often simply use $\lim(\mathfrak{c}^{\bullet}_0, \ldots , \mathfrak{c}^{\bullet}_{n-1})$ to denote the global limit type $\limtp_U(\mathfrak{c}^{\bullet}_0, \ldots , \mathfrak{c}^{\bullet}_{n-1})$.
\end{definition}


\subsection{Invariant Types}

Let $A \subseteq D \subseteq U$ with $A$ small.

\begin{definition}\label{def:invariant}
	A type $p \in S_D(x)$ is \emph{invariant over $A$} if for all $\phi(x, d) \in p$ and all $d' \in D$ such that $d' \equiv_A d$, we have $\phi(x, d') \in p$. Moreover, if we do not wish to specify the invariance base, we may simply say $p\in S_D(x)$ is \emph{invariant} to indicate that it is invariant over some small subset of $D$.
\end{definition}

We are mostly interested in invariant types that are global. Indeed, the name is suggestive of the fact that a global type which is invariant over $A$ is invariant under any global automorphism that fixes $A$ pointwise. It is important to note that not every local invariant type can be extended to a global invariant type without changing the invariance base. For example, if $L = \varnothing$ and $T$ is the theory of infinite sets, then given $a\in U$, it follows that $\tp_{\{a \}}(a)$ is invariant over $\varnothing$ but has no extension which is also invariant over $\varnothing$.

\begin{fact}\label{fact:finsat}
	If $p \in S_D(x)$ is finitely satisfiable in $A$, then it is invariant over $A$.
\end{fact}


\begin{cor}
	Suppose $\II$ is a sequence of tuples from $D$. If it exists, $\limtp_D(\II)$ is invariant over $\II$.
\end{cor}

Finitely satisfiable (partial) types can always be extended invariantly without changing the invariance base. 

\begin{fact}
	If $\Gamma\subseteq  L_U(x)$ is finitely satisfiable in $A$, then $\Gamma$ extends, not necessarily uniquely, to a global type which is finitely satisfiable in $A$.
\end{fact}

Let $\kappa = \aleph_0 +|A|$. Suppose $\MM\models T$ is $\kappa^+$-saturated with $A\subseteq M\subseteq D$.

\begin{fact}
	If $p\in S_M(x)$ is invariant over $A$, then there is a unique type $p\uhr^D$ in $S_D(x)$ which extends $p$ while remaining invariant over $A$.
\end{fact}

The above facts are well-known and therefore stated without proof. A more detailed discussion can be found on pages 18 and 19 of \cite{simon:guidenip} at the beginning of Subsection 2.2.

\subsection{Morley Sequences}
Let $A\subseteq D\subseteq U$ with $A$ small, and let $\II = (b_i\colon i\in I)\subseteq D$ be a sequence of tuples indexed by $(I,<)$ an infinite linear order.

\begin{fact}\label{fact:Morley}
	If $p \in S_D(x)$ is invariant over $A$ and each 
	$b_i \models p \dhr_{A\,\cup\,\{b_j\,:\,j\,<\,i\}}$,
	then $\II$ is indiscernible over $A$ and $\tp_A^{\EM}(\II)$ is completely determined by $p$.
\end{fact}


\noindent This follows easily by induction. For an alternative presentation of Morley sequences which introduces product types first, see page 21 of \cite{simon:guidenip}.

\begin{definition}
	Given $\II$ as in Fact \ref{fact:Morley}, we call any infinite ordered sequence $\JJ \subseteq U$ such that $\JJ \equiv_A^{\EM} \II$ a \emph{Morley sequence for $p$ over $A$.}
\end{definition}

\begin{note}
	We often use the convention of adding an asterisk to reverse an ordering. For example, in the following lemma, we use $\JJ^*$ to denote $\JJ$ in the reverse order, so $b$ precedes $b'$ in $\JJ^*$ if and only if $b'$ precedes $b$ in $\JJ$.
\end{note}

\begin{lemma}\label{lem:Morley sequence for ultralimit}
Given an endless sequence $\mathcal{J} \subseteq D$ with the same tuple size as $\II$, if  $\mathcal{I} + \mathcal{J}^*$ is indiscernible over $A$ and $\limtp_D(\JJ)$ exists, then $\mathcal{I}$ is a Morley sequence for $\limtp_D(\JJ)$ over $A\mathcal{J}$.
\end{lemma}
\begin{proof}
Choose $b \in \mathcal{I}$, and let $\mathcal{I}' \subseteq \mathcal{I}$ be the largest initial segment of $\mathcal{I}$ that does not contain $b$. By indiscernibility, it follows that $\tp_{A\mathcal{J}\mathcal{I}'} (b) = \limtp_{A\mathcal{J}\mathcal{I}'} (\mathcal{J}).$
\end{proof}

\subsection{Product Types}
Let $A\subseteq D\subseteq U$ with $A$ small, and let $\kappa = \aleph_0 +|A|$. Suppose $\MM\models T$ is $\kappa^+$-saturated with $A\subseteq M\subseteq D$.

\begin{definition}\label{def:product}
	Let $p\in S_D(x)$ and $q \in S_M(y)$. Suppose $q$ is invariant over $A$.
	We define the \emph{product} $p \otimes q \in S_D(x,y)$ as follows: For all $\phi(x,y,z) \in  L$, where the size of $z$ may vary, and all $d \in D^{\abs{z}}$, we have
	\[
	\phi(x,y,d) \in p \otimes q \qquad \Longleftrightarrow  \qquad \forall\, a \models p \dhr_{Ad} \ \ \phi (a,y,d) \in q\uhr^U.
	\]
\end{definition}

One can easily check that the above product is a complete type and that the product operation is associative. If, in addition, $p$ is invariant over $A$, then the product is also invariant over $A$. 

\begin{note}
	We choose to resolve products from left to right. The reader should be aware that some authors resolve finite products from right to left (i.e., $\phi(x,y,d) \in p\otimes q \Leftrightarrow \forall b\models q\dhr_{Ad}\, \phi(x,b,d)\in p\uhr^U$) but then resolve infinite products from left to right. We, however, find it easier to keep the same order for both.
\end{note}

\begin{definition}
	Suppose $p(x),q(y)\in S_M$ are invariant over $A$. We say $p$ and $q$ commute if $p(x) \otimes q(y) = q(y) \otimes p(x)$.
\end{definition}

\begin{lemma}\label{lem:realizedtypescommute}
	If $p\in S_M(x)$ is realized in $M$ and $q\in S_M(y)$ is invariant over $A$, then $p$ and $q$ commute. 
\end{lemma}

\begin{proof}
	Let $a\in M$ realize $p$. Given $\phi(x,y,d)\in  L_M$, we have
	\begin{align*}
		\phi(x,y,d)\in p \otimes q \quad & \Longleftrightarrow \quad \phi(a,y,d)\in q \\
		& \Longleftrightarrow\quad  \forall b\models q \dhr_{Aad} \; \phi(x,b,d)\in p \\
		& \Longleftrightarrow\quad  \phi(x,y,d)\in q \otimes p .
	\end{align*}
\end{proof}

\begin{definition}
	If $p\in S_M(x)$ is invariant over $A$ and $n > 0$, we use $p^n$ to denote the \emph{$n$-fold product of $p$} given by 
	\[
	p^n = p(x_0) \otimes \cdots \otimes p(x_{n-1}) \in S_M( \abs{x} \cdot n).
	\]
	Furthermore, we define the \emph{$\omega$-fold product of $p$} as
	\[
	p^{\omega} = \bigcup_{0 < n < \omega} p^n \in S_M(\abs{x} \cdot \omega).
	\]
\end{definition}

\noindent Notice that if $\II$ is a Morley sequence for $p$ over $A$, then $\II \models^{\EM} p^{\omega} \dhr_A$.

\begin{lemma} \label{lem:morley}
	Given a sequence 
	\[
	\II+\JJ +\KK = (b_i\colon i\in I) + (b_j\colon j\in J) +(b_k\colon k\in K)\subseteq U	
	\]
	where $I+J+K$ is a partition of a linear order and $(I,<)$ has no maximum element, suppose $\II+\KK$ is indiscernible over $A$ and $p =\limtp_{A\II\JJ\KK}(\II)$ exists. If $\JJ^*$ is a Morley sequence for $p$ over $A\II\KK$, then $\II +\JJ+\KK$ is  indiscernible over $A$.
\end{lemma}

\begin{proof}
	Fix $\phi \in  L_A(x_0, \ldots, x_{n-1})$ and $\ell_0 < \cdots < \ell_{n-1} \in I$. We claim that for all $r,s,t<\omega$ such that $r+s+t=n$, if 
	\[
	i_0<\cdots <i_{r-1}\in I, \; j_0 < \cdots < j_{s-1} \in J, \text{ and } k_0<\cdots <k_{t-1}\in K,
	\] 
	then 
	\[
	\UU \models \phi(b_{\ell_0}, \ldots, b_{\ell_{n-1}}) \leftrightarrow \phi(b_{i_0}, \ldots, b_{i_{r-1}}, b_{j_0}, \ldots, b_{j_{s-1}}, b_{k_0}, \ldots, b_{k_{t-1}}).
	\]
	We proceed by induction on $s$.
	
	\bigskip
	\noindent \underline{$s=0$}: Our claim holds since $\II+\KK$ is indiscernible over $A$.
	\bigskip
	
	\noindent \underline{$s> 0$}: Suppose the claim holds for $s-1$. It follows that for all $m\in I$ such that $m> i_{r-1}$, we have
	\[
	\UU \models \phi(b_{\ell_0}, \ldots, b_{\ell_{n-1}}) \leftrightarrow \phi(b_{i_0}, \ldots, b_{i_{r-1}}, b_{m}, b_{j_1}, \ldots, b_{j_{s-1}}, b_{k_0}, \ldots, b_{k_{t-1}}).
	\]
	Since $p$ is the limit type of $\II$, we have 
	\[
	\phi(b_{\ell_0}, \ldots, b_{\ell_{n-1}}) \leftrightarrow \phi(b_{i_0}, \ldots, b_{i_{r-1}}, x, b_{j_1}, \ldots, b_{j_{s-1}}, b_{k_0}, \ldots, b_{k_{t-1}})\in p,
	\]
	so our claim holds since $\JJ^*$ is a Morley sequence for $p$ over $A\II\KK$.
\end{proof}

\begin{lemma}\label{lem:limmutind}
	Suppose $T$ is NIP. If a collection $\left( \II_i:i<n\right)$ of infinite sequences is mutually indiscernible and $\phi (x_0, \ldots , x_{n-1})\in \lim(\II_0)\otimes \cdots \otimes \lim(\II_{n-1})$, then there are end segments $\II_i' \subseteq \II_i$ such that all $\bar{a} \in \II'_0\times \cdots \times \II'_{n-1}$ realize $\phi$.
\end{lemma}

\begin{proof}
	The lemma clearly holds when $n=1$. Assume it holds for some $n\geq 1$ but fails for the collection $(\II, \JJ_0 , \ldots , \JJ_{n-1})$. It follows by Lemma \ref{lem:realizedtypescommute} that none of the indices $I,J_0,\ldots,J_{n-1}$ has a maximum element.
	Let $p(x)=\lim(\II)$ and $q(y_0, \ldots , y_{n-1})=\lim(\JJ_0)\otimes \cdots \otimes \lim(\JJ_{n-1})$.
	Let $\phi(x, \bar{y}, d)\in p\otimes q$, where $\phi(x, \bar{y}, z)\in  L$ and $d\in U$, witness the failure of the lemma.
	Let $\JJ=\JJ_0\times \cdots \times \JJ_{n-1}$, and let $a \models p \dhr_{\II \JJ d}$.
	Since $\phi(x, \bar{y}, d)\in p\otimes q$, it follows that $\phi(a, \bar{y}, d) \in q$.
	Since the lemma holds for $n$, there is an end segment $\JJ'$ of $\JJ$ (i.e., $\JJ'=\JJ'_0\times \cdots \times \JJ'_{n-1}$ with each $\JJ'_i$ an end segment of $\JJ_i$)
	such that all $\bar{b} \in \JJ' $ realize $\phi(a, \bar{y}, d)$.
	It follows that $\phi(x, \bar{b}, d)\in p$ for all $\bar{b}\in \JJ'$;
	therefore, for all such $\bar{b}$, there is an end segment $\II_{\bar{b}}\subseteq \II$ such that all elements of $\II_{\bar{b}}$ realize $\phi(x, \bar{b}, d)$.
	We use this to construct an indiscernible sequence which violates NIP.
	\bigskip
	\begin{description}[style=multiline,leftmargin=2.5cm,font=\normalfont]
		\item[\underline{Stage 0}] Let $\bar{b}_0\in \JJ'$ and $a_0\in \II_{\bar{b}_0}$.
		Let $\II^0 $ be an end segment of $\II$ excluding $a_0$, and let $\JJ^0$ be an end segment of $\JJ'$ excluding $\bar{b}_0$.
		\bigskip
		\item[\underline{Stage $2i+1$}] By our assumption, there is $a_{2i+1} \in \II^{2i}$ and $\bar{b}_{2i+1} \in \JJ^{2i}$ such that $\UU\models \neg \phi(a_{2i+1}, \bar{b}_{2i+1})$.
		Let $\II^{2i+1} $ be an end segment of $\II$ excluding $a_{2i+1}$, and let $\JJ^{2i+1}$ be an end segment of $\JJ'$ excluding $\bar{b}_{2i+1}$.
		\bigskip
		\item[\underline{Stage $2i+2$}] Let $\bar{b}_{2i+2} \in \JJ^{2i+1}$ and $a_{2i+2}\in \II_{\bar{b}_{2i+2}} \cap \II^{2i+1}$.
		Let $\II^{2i+2} $ be an end segment of $\II$ excluding $a_{2i+2}$, and let $\JJ^{2i+2}$ be an end segment of $\JJ'$ excluding $\bar{b}_{2i+2}$.
	\end{description}
	\bigskip
	The constructed sequence $\left( (a_i,\bar{b}_i):i<\omega \right)$ forces $\phi(x, \bar{y}, d)$ to have infinite alternation rank, contradicting NIP.
\end{proof}

\begin{cor}
	Suppose $T$ is NIP. If $\II$ and $\JJ$ are infinite mutually indiscernible sequences, then $\lim(\II)$ and $\lim(\JJ)$ commute.
\end{cor}
\begin{cor}
	Suppose $T$ is NIP. If $\II$ is an indiscernible sequence with distinct Dedekind cuts $\cc_0,\ldots,\cc_{n-1}$, then 
	\[
	\lim(\cc_0^\bullet,\ldots , \cc_{n-1}^\bullet)= \lim(\cc_0^\bullet)\otimes\cdots \otimes \lim(\cc_{n-1}^\bullet)
	\] 	
	where each $\cc_i^\bullet\in\{ \cc_i^-,\cc_i^+\}$.
\end{cor}

\subsection{Geometric Stability}
Throughout this subsection, we assume $T$ is stable. Let $D\subseteq U$ be a small set of parameters.

\begin{definition}
    We say a global type $p \in S_U$ is \emph{stationary over $D$} if it does not fork over $D$ and every other global extension of $p\dhr_D$ forks over $D$.
    Alternatively, if such is the case, we may simply say the local type $p \dhr_D$ is \emph{stationary}.
\end{definition}

\noindent If $q \in S_D$ is stationary, then $q\uhr^U$, as defined in Subsection \ref{ss:types and type spaces}, refers to its unique global nonforking extension.

\begin{fact}
Any type over a model is stationary.
\end{fact}

\noindent See \cite[Corollary 8.5.4]{tent:modeltheory}.

\begin{definition}
    Given $A, B \subseteq U$, we say that $A$ is \emph{forking independent} from $B$ over $D$, and write $A \dnf_D B$, if  for all finite tuples $a \in A^{<\omega}$, the type $\tp_{DB}(a)$ does not fork over $D$.
\end{definition}

For a review of forking, see Chapter 7 of \cite{tent:modeltheory}.

\begin{definition}
   We say a set of tuples $A\subseteq U^{<\omega}$ is \emph{independent over $D$} if \mbox{$a \dnf_D A\setminus a$} for all $a \in A$. 
    Otherwise, we say $A$ is \emph{dependent over $D$}.
\end{definition}

\begin{definition}\label{def:cycle}
    A finite set of tuples $A \subseteq U^{<\omega}$ is called a \emph{cycle over $D$} if it is dependent over $D$ and all its proper subsets are independent over $D$.
    Moreover, such a set is called an \emph{$n$-cycle over $D$} if its cardinality is $n$.
\end{definition}

We borrow the above terminology from the theory of matroids. See, for example, the definition of a cycle on page 133 of \cite{wilson:introductiontographtheory}.

\begin{lemma} \label{lem:stable types commute}
    If $p(x), q(y) \in S_U$ are invariant over $D$, then they commute.
\end{lemma}

\begin{proof}
If a formula $\phi(x,y,d) \in L_U$ is in $p\otimes q$ but not $q\otimes p$, then it defines a half graph on any Morley sequence $(a_i b_i : i < \omega) \models (p\otimes q)^\omega\dhr_{Dd}$.
\end{proof}

\begin{lemma}\label{lem:independent iff Morley}
    Given finitely many tuples $a_0, \dots, a_{n-1} \in U^{<\omega}$ and a small base $D \subseteq U$ such that each type $p_i = \tp_D(a_i)$ is stationary, it follows that $\bar{a}$ is independent over $D$ if and only if $\bar{a}\models \left(p_0\uhr^U \otimes \cdots \otimes p_{n-1}\uhr^U\right)\dhr_D.$
\end{lemma}
\begin{proof}
    ($\Rightarrow$): Stationarity.
    
    ($\Leftarrow$): Commutativity.
\end{proof}

\begin{lemma}(Cycle Extension)\label{lem:cycle extension}
    If $a_0,\dots, a_{n-1} \in U^{<\omega}$ form a cycle over $D$, then for every small $E\supseteq D$, there exists a cycle $\bar{b}$ over $E$ with $\bar{a} \equiv_D \bar{b}$.
\end{lemma}
\begin{proof}
    Let $p$ be a global nonforking extension of $\tp_D(\bar{a})$, and let $\bar{b} \models p\dhr_E$.
\end{proof}
 
 Let $E_0, \ldots, E_{n-1}$ be $\varnothing$-definable equivalence relations whose classes partition the $L$-sorts $X_0, \ldots, X_{n-1}$, respectively. 
 Let $Y_0, \ldots, Y_{n-1}$ be the corresponding imaginary sorts in $L^{\eq}$, and for $i<n$, let $\pi_i : X_i \to Y_i$ denote the projection taking a real tuple to the imaginary element representing its $E_i$-class.
 
\begin{lemma}\label{lem:imaginary cycles}
   Given an imaginary cycle $(b_0, \ldots, b_{n-1}) \in U^{\bar Y}$ over a small imaginary base $B \subseteq U^{\eq}$, there exists a real cycle $(a_0, \ldots, a_{n-1}) \in U^{\bar X}$ over a small real base $A \subseteq U$ such that
   \[
   \pi_0(a_0) \cdots \pi_{n-1}(a_{n-1}) \ \equivnew_B^{\eq} \ b_0 \cdots b_{n-1}.
   \]
\end{lemma}

\begin{proof}
    Choose $A \subseteq U$ such that $B \in \dcl(A)$. 
    By Lemma \ref{lem:cycle extension}, we may assume $\bar b$ is a cycle over $A$. 
    For the rest of the argument, we work in $(T^{\eq})_A$. 
    Construct $\bar a$ by induction as follows: given $a_0, \ldots, a_{i-1}$, choose $a_i$ so that $\pi_i(a_i) = b_i$ and 
    \[
    a_i \dnf_{b_i} a_0 \cdots a_{i-1} \bar b.
    \]
    Now, given $m<n$ and $i_0 < \cdots < i_{m-1} < n$, we have
    \[
    b_{i_{m-1}} \dnf_{\quad \ b_{i_0} \cdots \, b_{i_{m-2}}} a_0 \cdots a_{i_{m-2}}
    \]
    by monotonicity, transitivity, and symmetry, so we may obtain 
    \[
    b_{i_{m-1}} \dnf a_0 \cdots a_{i_{m-2}} \quad \text{and} \quad a_{i_{m-1}} \dnf a_0 \cdots a_{i_{m-2}},
    \]
    successively, using transitivity.
    Thus, every proper subset of $\bar a$ is independent.
    The fact that $\bar a$ is dependent follows from monotonicity.
\end{proof}

\section{Distality Rank}\label{s:distality rank}

Let  $\II$ be an indiscernible sequence $(b_i: i \in I) \subseteq U$ indexed by an infinite linear order $(I, <)$. Suppose $\II_0 + \cdots + \II_n$ is a partition of $\II$ corresponding to a Dedekind partition $I_0+ \cdots +I_n$ of $I$.
Let $A$ be a sequence $(a_0,\ldots,a_{n-1}) \subseteq U$. Assume $\abs{b_i}=\abs{a_j}$ for all $i\in I$ and $j<n$. 

\begin{definition}
    We say that $A$ \emph{inserts (indiscernibly) into} $\II_0+\cdots+\II_n$ if the sequence remains indiscernible after inserting each $a_i$ at the corresponding cut $\mathfrak{c}_i$, i.e., the sequence
    $$\II_0 + a_0 + \II_1 + a_1 + \cdots + \II_{n-1}+ a_{n-1}+ \II_n$$
    is indiscernible. Moreover, for any $A' \subseteq A$, we say that $A'$ \emph{inserts (indiscernibly) into} $\II_0+\cdots+\II_n$ if the sequence remains indiscernible after inserting each $a_i\in A'$ at the corresponding cut $\mathfrak{c}_i$. For simplicity, we may say that \emph{$A$ (or $A'$) inserts into $\II$} when the partition of $\II$ under consideration is clear.
\end{definition}

\begin{definition}
For $n>m>0$, we say that the Dedekind partition $\II_0+\cdots +\II_n$ is \emph{$m$-distal} if every sequence $A=(a_0, \ldots, a_{n-1}) \subseteq U$ which does not insert into $\II$ contains some $m$-element subsequence which does not insert into $\II$.
\end{definition}

\subsection{Distality Rank for EM-Types}

\begin{definition}
Given $n>m>0$, a complete EM-type $\Gamma$ is \emph{$(n,m)$-distal} if every Dedekind partition $\II_0+\cdots+\II_n$ with $\EM$-type $\Gamma$ is $m$-distal. 
\end{definition}


When considering the $(n,m)$-distality of a complete EM-type, the only interesting cases are those where $n=m+1$.

\begin{lemma} \label{lem:insert}
Fix $m>0$ and $\kappa$ a cardinal. Let $A= (a_\alpha: \alpha < \kappa) \subseteq U$, and let $\mathcal{C}=(\mathfrak{c}_\alpha :\alpha < \kappa)$ be a collection of Dedekind cuts of some infinite linear order $(I, <)$. Let $\II$ be an indiscernible sequence indexed by $I$, and let $\Gamma = \tp^{\EM}(\II)$. 
If $\Gamma$ is $(m+1,m)$-distal and all $m$-element (or smaller) subsets from $A$ insert into $\II$, each element $a_\alpha$ at the corresponding cut $\mathfrak{c}_\alpha$, then the entire sequence inserts into $\II$.
In particular, if $\Gamma$ is $(m+1,m)$-distal, then $\Gamma$ is $(n,m)$-distal for all $n>m$.
\end{lemma}

\begin{proof}
Suppose $\Gamma$ is $(m+1,m)$-distal and all $m$-element (or smaller) subsets from $A$ insert into $\II$. 
Let $P(\beta)$ assert that any ($m+1$)-element (or smaller) subset from the tail $A_{\geq \beta}=(a_\alpha: \beta \leq \alpha < \kappa)$ inserts into $\II \cup A_{<\beta}$ (i.e., the sequence created by inserting each element of $A_{<\beta}$ at its corresponding cut). 
We proceed by induction on $\beta$.
\bigskip

\begin{description}[style=multiline,leftmargin=1.5cm,font=\normalfont]
	\item[\underline{$\beta=0$}]  $P(0)$ holds since $\II \models^{\EM} \Gamma$  and $\Gamma$ is $(m+1,m)$-distal.
\bigskip
	\item[\underline{$\beta+1$}] Let $\JJ= \II\cup A_{<\beta+1}$.  $P(\beta)$ asserts that $\JJ$ is indiscernible and any $m$-element (or smaller) subset from the tail $A_{\geq \beta +1}$ inserts into $\JJ$. Thus, $P(\beta+1)$ holds since $\JJ \models^{\EM} \Gamma$  and $\Gamma$ is $(m+1,m)$-distal.
\bigskip
	\item[\underline{$\beta$ limit}]  Let $\JJ= \II\cup A_{<\beta}$.  Assume we have $\phi(\bar{c}) \nleftrightarrow \phi(\bar{d})$ witnessing that $P(\beta)$ does not hold, where $\bar{c}$ and $ \bar{d}$ have the same relative order in $\JJ \cup A'$ and $A'$ is an ($m+1$)-element (or smaller) subset of $A_{\geq \beta}$. Let $\beta'$ be the smallest ordinal such that $\bar{c}, \bar{d} \in \II\cup A_{<\beta'}\cup A'$. Now we have $\beta' < \beta$ and $\neg P(\beta')$.
\end{description}
\end{proof}

\begin{definition}
Given $m>0$, a complete EM-type is \emph{$m$-distal} if it is $(m+1,m)$-distal.
\end{definition}

\noindent Notice that for any $n>m>0$, if a complete EM-type is $m$-distal, then it is also $n$-distal.

\begin{definition}\label{def:drGamma}
The \emph{distality rank} of a complete EM-type $\Gamma$, written $\DR(\Gamma)$, is the least $m \geq 1$ such that $\Gamma$ is $m$-distal. If no such finite $m$ exists, we say the distality rank of $\Gamma$ is $\omega$.
\end{definition}

It is interesting to note that, given the generality of Lemma \ref{lem:insert}, it would make sense to define ($\beta,\alpha$)-distal and $\alpha$-distal for arbitrary, not only finite, ordinals $\beta>\alpha>0$. We could then define the distality rank of a complete EM-type as the least ordinal $\alpha$ for which it is $\alpha$-distal. However,  since any failure of a sequence to be indiscernible is witnessed by finitely many elements from that sequence, this yields only one infinite distality rank, namely $\omega$. Thus, the resulting definition of distality rank would be equivalent to Definition \ref{def:drGamma}.

\begin{definition}
	Fix $n> 0$, and let
	\[
	I_0 = \omega, \quad I_1 = \omega^* + \omega, \quad \ldots, \quad I_{n-1} = \omega^* + \omega, \quad I_n = \omega^*
	\]
	where $\omega^*$ is $\omega$ in reverse order. If $\II\subseteq U$ is a sequence indexed by $I = I_0 + \cdots + I_n$, we call the corresponding partition $\II_0+\cdots + \II_n$ an \emph{$n$-skeleton}. 
\end{definition}

\noindent Notice that an $n$-skeleton is a Dedekind partition with $n$ cuts.
\begin{prop} \label{prop:skeleton}
	Given $m>0$, a complete $\EM$-type $\Gamma$ is \emph{$m$-distal} if and only if there is an $(m+1)$-skeleton $\II_0+\cdots +\II_{m+1}\models^{\EM} \Gamma$ which is $m$-distal. 
\end{prop}

\begin{proof}
	($\Rightarrow$): The Standard Lemma \cite[Lemma 7.1.1]{tent:modeltheory} asserts that $\II \EMM \Gamma$ of the appropriate order type exists.
	
	($\Leftarrow$):  Suppose $\Gamma$ is not $m$-distal.  Let $\II=\II_0+\cdots + \II_{m+1} \models^{\EM} \Gamma$ be an $(m+1)$-skeleton. We will show that the skeleton is not $m$-distal. 
	
	Since $\Gamma$ is not $m$-distal, there exist $\JJ \models^{\EM} \Gamma$, a Dedekind partition $\JJ=\JJ_0 + \cdots + \JJ_{m+1}$, and a sequence $A = (a_0, \ldots, a_m) \in U$ such that all $m$-sized subsets insert but $A$ does not. Let $\phi \in \Gamma$ and $\bar{b}_i \in \JJ_i$ such that
	\[
	\mathcal{U} \not \models \phi(\bar{b}_0, a_0, \ldots, \bar{b}_m, a_m, \bar{b}_{m+1}).
	\]
	Construct $\sigma: \II \rightarrow \JJ$ an order-preserving map such that 
	\[
	\bar{b}_i \subseteq \sigma(\II_i) \subseteq \JJ_i.
	\]
	We can extend $\sigma$ to an automorphism of $\mathcal{U}$. Let 
	\[
	A' = (\sigma^{-1}(a_0), \ldots, \sigma^{-1}(a_m)).
	\] 
	Now any $m$-sized subset of $A'$ inserts into $\II_0 + \cdots + \II_{m+1}$, but $A'$ does not.
\end{proof}
\begin{definition} \label{def:witness}
	We say $(\phi,A,B)$ is a \emph{witness} that an indiscernible Dedekind partition $\II_0 + \cdots + \II_{m+1}$ is not $m$-distal if, as in the previous proof, the following hold:
	
	\begin{itemize}
		\item $\phi\in \tp^{\EM} (\II_0+\cdots + \II_{m+1}),$
		\item $A=(a_0,\ldots,a_m)\subseteq U$ is a sequence such that any proper subsequence inserts into the partition,
		\item $B=(\bar{b}_0, \ldots , \bar{b}_{m+1} ) $ where each $\bar{b}_i$ is a finite increasing sequence in $\II_i$, and
		\item $\mathcal{U} \not \models \phi(\bar{b}_0, a_0, \ldots, \bar{b}_m, a_m, \bar{b}_{m+1})$.
	\end{itemize}

\end{definition}
\begin{prop}\label{prop:dense}
	Let $m>0$ and $\Gamma \in S^{\EM}$. Suppose $\II = \II_0 + \cdots + \II_{m+1} \models^{\EM} \Gamma$ is a Dedekind partition whose underlying index is dense with no endpoints. If $\II$ is not $m$-distal, then every Dedekind partition $\JJ_0 + \cdots + \JJ_{m+1} \models^{\EM} \Gamma$ is not $m$-distal.
\end{prop}
\begin{proof}
	Let $\JJ = \JJ_0 + \cdots + \JJ_{m+1} \models^{\EM} \Gamma$ be a Dedekind partition. Suppose $\II$ is not $m$-distal. Without loss of generality, we may assume that $\II \subseteq U^t$ for some $t < \omega$. Let $(\phi, A, B)$ witness that $\II$ is not $m$-distal, and let $\sigma:B \to \JJ$ be an order-preserving map such that $\sigma(B \cap \II_i) \subseteq \JJ_i$ for each $i \leq m+1$.
	Given a finite $\JJ' \subseteq \JJ$, there exists an order preserving map $\tau : \JJ' \to \II$ such that $\tau \circ \sigma (B) = B$ and $\tau(\JJ' \cap \JJ_i) \subseteq \II_i$ for each $i \leq m+1$.
	Since any such $\tau$ extends to an automorphism of $\UU$, by compactness, there exists $A'$ such that $(\phi, A', \sigma(B))$ witnesses that $\JJ$ is not $m$-distal.
\end{proof}

We would like the distality rank of an indiscernible Dedekind partition to depend only on its EM-type. Unfortunately, Propositions \ref{prop:skeleton} and \ref{prop:dense} are not strong enough to preclude the existence of an EM-type whose densely ordered realizations are $m$-distal but whose discretely ordered realizations are not. This pathology can be eliminated in the NIP context (Theorem \ref{thm:order type does not matter for NIP}); however, before we can show this, we need a very technical but essential lemma which will be used to prove several theorems in this paper.

\begin{lemma}[Base Change Lemma]\label{lem:strongbasechange}
	Suppose $T$ is NIP and $m > 0$. If
	\begin{itemize}
		\item $\II = \II_0 + \cdots + \II_{m+1}$ is a Dedekind partition, 
		\item $A = (a_0, \ldots, a_m)$ is a sequence such that every proper subsequence inserts into $\II$, and
		\item $D \subseteq U$ is a small set of parameters,
	\end{itemize}
	then there is a sequence $A' = (a_0', \ldots, a_m')$ such that $A' \equiv_{\II} A$ and
	\[
		a'_{\sigma(0)} \cdots a'_{\sigma(m-1)} \models \limtp_D \left( \mathfrak{c}^-_{\sigma(0)}, \ldots, \mathfrak{c}^-_{\sigma(m-1)}\right)
	\]
	for each $\sigma: m \rightarrow m+1$ increasing.
	
\end{lemma}

\begin{proof}
	Assume no such $A'$ exists. By compactness, there are $\phi \in \tp_{\II}(a_0, \ldots, a_m)$ and $\psi_\sigma \in \limtp_D(\cc^-_{\sigma(0)}, \ldots, \cc^-_{\sigma(m-1)})$ for each $\sigma: m \rightarrow m+1$ increasing such that
	\begin{equation}
		 \phi(x_0, \ldots, x_m) \vdash \bigvee_\sigma \neg \psi_\sigma(x_{\sigma(0)}, \ldots, x_{\sigma(m-1)}). \label{eq:phipsi} \tag{$*$}
	\end{equation}
	First we handle the case where $\II$ is dense. Let $B\subseteq \II$ be the parameters of $\phi$.
	For each $\sigma$ as above, we construct an indiscernible sequence $\JJ_\sigma$ by induction: 
	\bigskip
	\begin{description}[style=multiline,leftmargin=2.5cm,font=\normalfont]
		\item[\underline{Stage 0}] For all $j < m + 1$, choose $\II_j^0$ to be a proper end segment of $\II_j$ excluding $B$ such that each $\psi_\sigma$ is satisfied by every element of $\II_{\sigma(0)}^0 \times \cdots \times \II_{\sigma(m-1)}^0$. Let $\II^0 = \II$, and let $\JJ^0_\sigma = \emptyset$ for each $\sigma$. 
		\bigskip
		\item[\underline{Stage $2i+1$}] Let $\II'$ be a finite subset of $\II^{2i}$ containing $B$. There is an increasing map 
		\[
			\II' \, \longrightarrow \,\, \II \setminus \bigcup_j \II_j^{2i}
		\] 
		fixing $B$ such that for each $j < m+1$, elements to the left of $\II_j^{2i}$ remain to the left and all other elements map to the right of $\II_j^{2i}$. This map extends to an automorphism fixing $B$, so by compactness, there is $A' = (a_0', \ldots, a'_m)$ realizing $\phi$ such that if we assign each $a'_j$ to the cut of $\II^{2i}$ immediately to the left of $\II^{2i}_j$, then any proper subsequence of $A'$ inserts into $\II^{2i} \supseteq \II$. By (\ref{eq:phipsi}), we can now choose $\sigma_i: m \rightarrow m+1$ increasing so that
		\[
			a'_{\sigma_i(0)} \cdots a'_{\sigma_i(m-1)} \not \models \psi_{\sigma_i}.
		\]
		Let 
		\[
			\II^{2i+1} = \II^{2i} \cup \left\{a'_{\sigma_i(j)} \, : \, j < m \right\}
		\] 
		where each $a'_{\sigma_i(j)}$ is inserted immediately to the left of $\II^{2i}_{\sigma_i(j)}$. Let 
		\[
			\JJ^{2i+1}_{\sigma_i} = \JJ^{2i}_{\sigma_i} + \left(a'_{\sigma_i(0)}, \ldots, a'_{\sigma_i(m-1)} \right).
		\] 
		For each $j < m+1$, let $\II^{2i+1}_j = \II^{2i}_j$.
		\bigskip
		\item[\underline{Stage $2i + 2$}] For each $j < m+1$, choose $b_j \in \II^{2i+1}_j $ and an end segment $\II^{2i+2}_j$ of $\II^{2i+1}_j$ excluding $b_j$. Let $\II^{2i+2} = \II^{2i+1}$, and for each $\sigma$, let 
		\[
			\JJ^{2i+2}_\sigma = \JJ_\sigma^{2i+1} + \left(b_{\sigma(0)}, \ldots, b_{\sigma(m-1)} \right).
		\]
	\end{description}
	\bigskip
	For each $\sigma$, let $\JJ_\sigma = \bigcup_{i < \omega} \JJ_\sigma^i$. Choose a $\sigma$ which appears infinitely many times in $(\sigma_i \, : \, i < \omega)$. It follows that $\psi_\sigma$ alternates infinitely many times on $\JJ_\sigma$, contradicting NIP.
	
	In the case where $\II$ is not dense, we may no longer assume the above construction can continue ad infinitum; however, finitely many stages will suffice. For $i<\omega$, notice that
	\[ \sum_\sigma \alt\left(\psi_\sigma, \JJ^i_\sigma\right) =i-1; \]
	thus, we only need to complete the construction through stage $2n+2$ where
	\[ n\geq \frac{\sum_\sigma \alt(\psi_\sigma)}{2} \] 
	to reach a contradiction.
	For $i<\omega $, let
	\[ s(i)=0+1+\cdots +i =\frac{i^2+i}{2}. \]
	Make the following modifications to Stage 0 and Stage $2i+1$.
	\bigskip
	\begin{description}[style=multiline,leftmargin=2.5cm,font=\normalfont]
		\item[\underline{Stage 0}] For each $j<m+1$, choose an increasing sequence
		\[ E_j=\left( e_j^0, \ldots, e_j^{s(n)-1} \right) \subseteq \II_j \]
		and a proper end segment $\II_j^0$ of $\II_j$ such that 
		\[ B\cap \II_j \, < \, E_j \, < \, \II_j^0 \] 
		and such that each $\psi_\sigma$ is satisfied by... (continue as above)
		\bigskip
		\item[\underline{Stage $2i+1$}] Let $\II'$ be a finite subset of 
		\[ \II^{2i}\setminus \left\{e_j^0, \ldots, e_j^{s(i)-1} : j<m+1\right\} \]
		containing $B$. There is an increasing map \[\II'\rightarrow \II\setminus \bigcup_j \left( \II_j^{2i} \cup \left\{ e_j^0, \ldots, e_j^{s(i-1)-1} \right\} \right) \]
		fixing $B$ such that... (continue as above)
	\end{description}
\end{proof}

Now we can show that, in an NIP context, the $m$-distality of an indiscernible Dedekind partition depends only on its EM-type.

\begin{theorem} \label{thm:order type does not matter for NIP}
	Suppose $T$ is NIP. Given $m>0$, a complete EM-type $\Gamma$ is $m$-distal if and only if there is a Dedekind partition $\II_0 + \cdots + \II_{m+1} \models^{\EM} \Gamma$ which is $m$-distal.
\end{theorem}
\begin{proof}
	($\Leftarrow$): Suppose $\Gamma \in S^{\EM}$ is not $m$-distal. Let $J = (m+2)\times \mathbb{Q}$ be lexicographically ordered, and let $J_i = \{i\}\times \mathbb{Q}$ for $i \leq m+1$.
	By the Standard Lemma (\cite[Lemma 7.1.1]{tent:modeltheory}), there exists $\JJ = (b_j : j \in J) \models^{\EM} \Gamma$.
	Let $K = K_0 + \cdots + K_{m+1}$ with $K_0 = \{0\}\times \mathbb{Z}^{\geq0}$, $K_{m+1} = \{m+1\}\times \mathbb{Z}^{\leq0}$, and $K_i = \{i\} \times \mathbb{Z}$ for $0<i<m+1$.
	Let $\KK = (b_k : k \in K)$. Since $\KK$ is a skeleton, there is $(\phi, A,B)$ witnessing that $\KK$ is not $m$-distal, and by the Base Change Lemma (Lemma \ref{lem:strongbasechange}), there is $A' \equiv_\KK A$ such that
	\[ a'_{\sigma(0)} \cdots a'_{\sigma(m-1)} \models \limtp_\JJ (\cc^-_{\sigma(0)}, \dots, \cc^-_{\sigma(m-1)}).\]
	It follows that $(\phi, A', B)$ witnesses that $\JJ$ is not $m$-distal, so we can apply Proposition \ref{prop:dense} to conclude that no Dedekind partition $\II_0 + \cdots + \II_{m+1} \models^{\EM} \Gamma$ is $m$-distal.
\end{proof}


\subsection{Distality Rank for Theories}

\begin{definition}
Given $m > 0$, a theory $T$, not necessarily complete, is \emph{$m$-distal} if for all completions of $T$ and all tuple sizes $\kappa$, every $\Gamma \in S^{\EM}(\kappa \cdot \omega)$ is $m$-distal.
\end{definition}

\noindent In the existing literature, an NIP theory is called distal if and only if it is 1-distal (see \cite[Definition 2.1]{simon:distalandnondistal} and following).

\begin{definition}\label{drT}
	The \emph{distality rank} of a theory $T$, written $\DR(T)$, is the least $m \geq 1$ such that $T$ is $m$-distal. If no such $m$ exists, we say the distality rank of $T$ is $\omega$.
\end{definition}

Adding named parameters to a theory does not increase its distality rank.
\begin{prop} \label{prop:DRA}
If $T$ is a complete theory and $B\subseteq U$ is a small set of parameters, then $\DR (T_B) \leq \DR (T)$.
\end{prop}

\begin{proof}
Let $\II= \left( a_i:i\in I \right) \subseteq U$ be a sequence of tuples, and let $(b_\alpha :\alpha <\kappa)$ be an enumeration of the base $B$.
Notice that $\II$ is indiscernible in $T_B$ if and only if the sequence $$(a_i+(b_\alpha: \alpha <\kappa):i\in I)$$ is indiscernible in $T$.
Thus, given $m>0$, if $T$ is $m$-distal, then $T_B$ is also $m$-distal.
\end{proof}

In an NIP context, the distality rank of a theory is completely unaffected by base changes.

\begin{theorem} [Base Change Theorem] \label{thm:basechange}
	If $T$ is NIP and $B\subseteq U$ is a small set of parameters, then $\DR(T_B) = \DR(T)$.
	
\end{theorem}

\begin{proof}
Proposition \ref{prop:DRA} asserts that $\DR(T_B)\leq \DR(T)$; thus, it suffices to show that for $m>0$, if $T_B$ is $m$-distal, then $T$ is also $m$-distal. 

Suppose $\Gamma \in S^{\EM}$ is not $m$-distal. By the Standard Lemma \cite[Lemma 7.1.1]{tent:modeltheory}, there is a skeleton $\II_0 + \cdots + \II_{m+1} \models^{\EM} \Gamma$ which is indiscernible over $B$. Furthermore, Proposition \ref{prop:skeleton} asserts that this skeleton is not $m$-distal; thus, there exists a sequence $A = (a_0, \ldots, a_m)$ such that every proper subsequence inserts indiscernibly over $\varnothing$ but $A$ does not. Applying the Base Change Lemma (Lemma \ref{lem:strongbasechange}) with $D = B \cup \II$ yields a sequence $A'$ such that every proper subsequence inserts indiscernibly over $B$ but $A'$ does not.
\end{proof}

We conclude this section with the easy observation that 1-distal theories are unstable.

\begin{prop}\label{prop:stablenotdistal}
	If $T$ is stable, then $\DR(T)\geq 2$.
\end{prop}
\begin{proof}
	Let $\II = \II_0 + \II_1 + \II_2$ be a nonconstant indiscernible skeleton. There is $a\in U$ which inserts at $\cc_0$. Since $T$ is stable, $\II$ is totally indiscernible, so $a$ also inserts at $\cc_1$. It follows that $(x_0\neq x_1, (a,a), \varnothing )$ witnesses that the skeleton is not 1-distal.
\end{proof}

\section{Strong Distality Rank}\label{s:strong distality}
\begin{definition}
	Given $m>0$, an indiscernible Dedekind partition $\II_0 + \II_1$ is \emph{strongly $m$-distal} if for all $a\in U$ and all sequences of small sets $\bar{D}=(D_0, \ldots, D_{m-1})$ such that $\II_0+\II_1$ is indiscernible over $\bar{D}$ and $\II_0+a+\II_1$ is indiscernible over $\bigcup_{i \neq j}D_i$ for all $j<m$, the sequence $\II_0+a+\II_1$ is indiscernible over $\bar{D}$.
\end{definition}

\begin{lemma}\label{lem:strongly nm distal}
	Let $n \geq m>0$, and let $\II_0 + \II_1$ be a strongly $m$-distal Dedekind partition. Given $a\in U$ and a sequence of small sets $\bar{D}=(D_0, \ldots, D_{n-1})$, if $\II_0+\II_1$ is indiscernible over $\bar{D}$ and $\II_0+a+\II_1$ is indiscernible over $D_{i_0} \cdots D_{i_{m-2}}$ for all $i_0 < \cdots < i_{m-2} < n$, then $\II_0+a+\II_1$ is indiscernible over $\bar{D}$.
\end{lemma}

\begin{proof}
    We proceed by induction. Suppose the result holds for some $n\geq m$. Given $a \in U$ and a sequence of small sets $(D_0, \ldots, D_n)$, if $\II$ is indiscernible over $\bar{D}$ and $\II_0 + a + \II_1$ is indiscernible over $D_{i_0} \cdots D_{i_{m-2}}$ for all $i_0 < \cdots < i_{m-2} < n+1$, then $\II_0+a+\II_1$ is indiscernible over $D_{j_0} \cdots D_{j_{m-2}}D_n$ for all $j_0 < \cdots < j_{m-2} < n$ since $\II_0 + \II_1$ is strongly $m$-distal. For each $j<n$, let $E_j = D_j D_n$. Now $\II_0+\II_1$ is indiscernible over $\bar{E}$ and $\II_0+a+\II_1$ is indiscernible over $E_{j_0} \cdots E_{j_{m-2}}$ for all $j_0 < \cdots < j_{m-2}<n$, so our hypothesis implies that $\II_0+a+\II_1$ is indiscernible over $\bar{E}$.
\end{proof}

\subsection{Strong Distality Rank for EM-Types}

\begin{definition}
	Given $m>0$, a complete $\EM$-type $\Gamma$ is \emph{strongly $m$-distal} if all Dedekind partitions $\II_0+\II_1\EMM \Gamma$ are strongly $m$-distal.
\end{definition}

\begin{definition}
	The \emph{strong distality rank} of a complete EM-type $\Gamma$, written $\SDR(\Gamma)$, is the least $m \geq 1$ such that $\Gamma$ is strongly $m$-distal. If no such finite $m$ exists, we say the strong distality rank of $\Gamma$ is $\omega$.
\end{definition}

\begin{lemma}\label{lem:witness}
	Let $m>0$. Suppose $\Gamma \in S^{\EM}$ is not strongly $m$-distal and $\II = \II_0 + \II_1 \EMM \Gamma$ is a Dedekind partition indexed by $(I_0 + I_1, <)$. There is a \underline{witness} $(\bar{D}, \phi, a)$ where
	\begin{itemize}
		\item $\bar{D}=(D_0, \ldots , D_{m-1})$ is such that $\II$ is indiscernible over $\bar{D}$,
		\item $\phi(x) \in \tp_{\bar{D}}^{\EM}(\II)$, and
		\item $a \in U$ is such that $\II_0 + a + \II_1$ is indiscernible over $\bigcup_{i \neq j}D_i$ for all $j<m$ but $\UU \not \models \phi(a)$.
	\end{itemize}
	Moreover, we may assume that $\bar{D}=(Bd_0, \ldots , Bd_{m-1})$ for some finite base $B\subseteq U$ and singletons $d_0,\ldots , d_{m-1} \in U^1$ and that $\II_0+a+\II_1$ is indiscernible over $B\cup\{d_i\colon i\neq j \}$ for each $j<m$. 
\end{lemma}

\begin{proof}
    Let $\JJ = \JJ_0 + \JJ_1 \EMM \Gamma$ be a Dedekind partition which is not strongly $m$-distal. Choose $D_0, \dots, D_{m-1}$ with $|D_0|+ \cdots + |D_{m-1}|$ minimal such that $\JJ$ is indiscernible over $\bar{D}$ and we can find 
	\begin{itemize}
		\item $\phi \in \tp_{\bar{D}}^{\EM}(\JJ)$,
		\item $a \in U$, and
		\item $\bar{b}_i$ increasing in $\JJ_i$
	\end{itemize}
	with $\UU \not \models \phi(\bar{b}_0, a, \bar{b}_1)$ and $\JJ_0 + a + \JJ_1$ indiscernible over each $\bar{D}\setminus D_i$.
	
	For each $i < m$, choose $d_i \in D_i^1$, and let $B = \bar{D}\setminus d_0\cdots d_{m-1}.$ Let $\JJ_0'$ be an end segment of $\JJ_0$ which completely excludes $\bar{b}_0$, and let $\JJ_1'$ be an initial segment of $\JJ_1$ which completely excludes $\bar{b}_1$. Now the sequence $\JJ' = \JJ_0' + \JJ_1'$ is indiscernible over $B\bar{b}_0\bar{b}_1\bar{d}$ and satisfies $\Gamma$. By compactness, we may assume each $\JJ_i'$ is indexed by $I_i$. Let $\sigma: \JJ' \rightarrow \II$ preserve indices. Since $\sigma$ extends to an automorphism, it follows that 
	\[
	\left( (\sigma(B\bar{b}_0\bar{b}_1 d_0), \ldots , \sigma(B\bar{b}_0\bar{b}_1 d_{m-1}) ), \; \; \sigma(\phi)  (\sigma(\bar{b}_0), x, \sigma(\bar{b}_1)  ), \; \; \sigma(a) \right)
	\] 
	is the desired witness. (Here we use $\sigma(\phi)$ is denote the formula created from $\phi$ by substituting $\sigma(b)$ for each named parameter $b\in B\bar{d}$ mentioned by $\phi$.)
\end{proof}

\begin{cor}\label{cor:SDR otype}
	Given $m>0$, a complete $\EM$-type $\Gamma$ is strongly $m$-distal if and only if there is a Dedekind partition $\II_0 + \II_1 \EMM \Gamma$ which is strongly $m$-distal.
\end{cor}

\begin{prop} \label{prop:strongdistal}
	Let $m>0$. Suppose a complete $\EM$-type $\Gamma$ is strongly $m$-distal. If a Dedekind partition $\II_0+\cdots +\II_{m+1}\EMM \Gamma$ is indiscernible over some small $B\subseteq U$ and $A=(a_0, \ldots , a_m)\subseteq U$ is such that every proper $A' \subset A$ inserts indiscernibly over $B$, then $A$ inserts indiscernibly over $B$. In particular, $\Gamma$ is $m$-distal.
\end{prop}
\begin{proof}
	Given a Dedekind partition $\II_0+ \cdots +\II_{m+1}\EMM \Gamma$ indiscernible over $B$, suppose every proper $A'\subset A$ inserts indiscernibly over $B$. Let $D_i=B\II_i a_i$ for each $i<m$. By strong $m$-distality, it follows that $\II_m + a_m +\II_{m+1}$ is indiscernible over $\bar{D}$.
\end{proof}

\begin{cor}\label{cor:DR <= SDR EM}
    If $\Gamma$ is a complete EM-type, then $\DR(\Gamma) \leq \SDR(\Gamma).$
\end{cor}

It is important to note that there are $\EM$-types for which this inequality is strict. 
A very nice example, suggested by Pierre Simon, can be found while working in the theory of the ordered random $m$-partite hypergraph ($\ORPG_m$). In this theory, the EM-type of any indiscernible sequence of singletons has distality rank 1 but strong distality rank $m$. We will discuss this example in more detail at the end of Subsection \ref{ss:ORPG}.

On the other hand, \cite[Lemma 2.7]{simon:distalandnondistal} implies that there is an important case where both ranks must agree.

\begin{fact}
    Suppose $T$ is NIP. 
    If $\Gamma$ is a complete EM-type and $\DR(\Gamma) = 1$, then $\SDR(\Gamma) = 1$. 
\end{fact}

\subsection{Strong Distality Rank for Theories}

\begin{definition}
	Given $m>0$, a theory $T$, not necessarily complete, is \emph{strongly m-distal} if for all completions of $T$ and all tuple sizes $\kappa$, every $\Gamma \in S^{\EM}(\kappa \cdot \omega)$ is strongly $m$-distal.
\end{definition}

\begin{definition}
	The \emph{strong distality rank} of a theory $T$, written $\SDR(T)$, is the least $m \geq 1$ such that $T$ is strongly $m$-distal. If no such $m$ exists, we say the strong distality rank of $T$ is $\omega$.
\end{definition}

\begin{prop}
	If $T$ is a complete theory and $B\subseteq U$ is a small set of parameters, then $\SDR (T_B) \leq \SDR (T)$.
\end{prop}
\begin{proof}
Similar to the proof of Proposition \ref{prop:DRA}.
\end{proof}

\begin{prop}\label{prop:boundsdr}
	If $T$ is an $L$-theory with quantifier elimination and $L$ contains no atomic formula with more than $m$ free variables, then $\SDR(T) \leq m$.
\end{prop}

\begin{proof}
	Let $b\in U^n$ and $d_0,\ldots, d_{m-1} \in U^1$. Suppose $\II = \II_0+\II_1 \subseteq U^\ell$ is Dedekind and indiscernible over $b\bar{d}$. Given $\phi \in  L(\ell+m+n)$, there is a $T$-equivalent formula 
	\[
	\bigvee_i\bigwedge_j \theta_{i,j} \left(x_{\sigma_{i,j}(0)},\ldots,x_{\sigma_{i,j}(m-1)} \right)
	\]
	where each $\theta_{i,j}$ is basic (i.e., an atomic formula or its negation) and each $\sigma_{i,j}: m \rightarrow \ell+m+n$ is a function. Thus, if $a\in U^\ell$ is such that $\II_0+a+\II_1$ is indiscernible over $b\bar{d}\setminus d_i$ for each $i<m$, then
	\[
		\UU\models \phi(a,\bar{d},b)\leftrightarrow \phi(a',\bar{d},b)
	\]
	for each $a'\in \II$. In light of Lemma \ref{lem:witness}, we conclude $T$ is strongly $m$-distal. 
\end{proof}

\begin{cor}\label{cor:bound SDR}
	Suppose $L$ is a language where all function symbols are unary and all relation symbols have arity at most $m \geq 2$. If $T$ is an $L$-theory with quantifier elimination, then $\DR(T) \leq \SDR(T) \leq m$.
\end{cor}
\begin{proof}
    Proposition \ref{prop:boundsdr} and Corollary \ref{cor:DR <= SDR EM}.
\end{proof}
\noindent We use this result in the next section to generate examples of theories with finite distality rank.

\section{Examples}\label{s:examples}

It appears that we have an infinite hierarchy which classifies theories by distality rank. 
We would like to show that this hierarchy is non-trivial by finding examples of theories which have distality rank $m$ for each $m \geq 1$.
Many examples of theories with distality rank 1 are listed in \cite{simon:distalandnondistal}. Among them are all o-minimal theories and the p-adics. We can quickly fill the rest of the finite ranks in the hierarchy using random graphs and hypergraphs.

\subsection{Random Graphs and Hypergraphs}

Fix $m \geq 2$, and let $L = \{\edge\}$ where $\edge$ is an $m$-ary relation symbol. 
Let $\RG_m$ denote the theory of the \emph{random $m$-uniform hypergraph} with hyperedge relation $\edge$. This structure is the Fra\"iss\'e limit of the class of all finite $m$-uniform hypergraphs, i.e., finite $L$-structures satisfying (2), below. 
Since Fra\"iss\'e limits are ultrahomogeneous (see \cite[Theorem 6.1.2]{hodges:shortermodeltheory}), $\RG_m$ clearly asserts the following schema:
\begin{enumerate}
    \item The domain, or \emph{vertex set} $\mathbb{V}$, is infinite.
    \item The relation $\edge$ is irreflexive and symmetric; i.e., it can be thought of as a collection of unordered sets,  or \emph{hyperedges}, each containing $m$ distinct vertices.
    \item For all $s,t < \omega$, if $A_0, \ldots, A_s, B_0, \ldots, B_t$ are distinct subsets in $[\mathbb{V}]^{m-1}$, then there is a vertex $d \in \mathbb{V}$ such that 
    \[\bigwedge_{i\leq s} A_i \edge d\ \land\ \bigwedge_{j\leq t} B_j \noedge d.\]
\end{enumerate}
Furthermore, a simple back-and-forth argument shows that this schema is countably categorical; thus, by Vaught's Test \cite[Theorem 2.2.6]{marker:modeltheory},  it is a complete axiomatization of $\RG_m$.
Since the theory of a Fra\"iss\'e limit has quantifier elimination (see \cite[Theorem 6.4.1]{hodges:shortermodeltheory}), Corollary \ref{cor:bound SDR} asserts that $\RG_m$ is $m$-distal.
However, it is not ($m-1$)-distal since, by compactness, there is an $(m-1)$-skeleton $\II=\II_0 + \cdots + \II_m \subseteq U^1$ along with elements $a_0,\ldots,a_{m-1} \in U^1$ such that the only hyperedge with vertices among $\II \cup \{a_0, \ldots, a_{m-1}\}$ is $\bar a$.
Therefore, we conclude that $\DR(\text{RG}_m) = m$.

\subsection{Random Partite Graphs and Hypergraphs}
Fix $m \geq 2$, and let $L = \{\edge, P_0, \ldots, P_{m-1}\}$ where $\edge$ is an $m$-ary relation symbol and each $P_i$ is a unary predicate symbol. 
Let $\text{RPG}_m$ denote the theory of the \emph{random $m$-partite hypergraph} with hyperedge relation $\edge$ and colors $P_0, \ldots, P_{m-1}$.
This structure is the Fra\"iss\'e limit of the class of all finite $L$-structures satisfying axioms (3) and (4), below. Using methods similar to those used for $\RG_m$, above, we may conclude that $\RPG_m$ has quantifier elimination and can be axiomatized by the following schema:
\begin{enumerate}
    \item The domain, or \emph{vertex set} $\mathbb{V}$, is infinite.
    \item Each $P_i$, often called a \emph{color}, contains infinitely many vertices.
    \item The colors partition the vertex set; i.e., $\mathbb{V} = P_0 \sqcup \cdots \sqcup P_{m-1}.$
    \item The hyperedge relation $\edge$ is a subset of $P_0 \times \cdots \times P_{m-1}$.
    \item Given $\ell < m$ and $s,t<\omega$, if $\bar{a}_0, \dots, \bar{a}_s$, $\bar{b}_0, \dots, \bar{b}_t$ are distinct tuples in $\prod_{k\neq \ell} P_k$, then there is a vertex $d \in P_\ell$ such that 
    \[\bigwedge_{i\leq s} \bar{a}_i \edge d\ \land\ \bigwedge_{j\leq t} \bar{b}_j \noedge d.\]
\end{enumerate}
Corollary \ref{cor:bound SDR} asserts that $\RPG_m$ is $m$-distal. 
However, by compactness, there is an indiscernible skeleton 
\[\II = \II_0 + \cdots + \II_m \subseteq \prod_{k<m} P_k(U)\]
along with tuples 
\[\bar{a}_0, \ldots, \bar{a}_{m-1} \in \prod_{k<m} P_k(U)\] such that the only hyperedge with vertices among $\II \bar{a}_0 \cdots \bar{a}_{m-1}$ is 
\[(\pi_0(\bar{a}_0), \ldots, \pi_{m-1}(\bar{a}_{m-1}))\]
where each $\pi_k : U^m \to U^1$ is the standard projection $(b_0, \ldots, b_{m-1}) \mapsto b_k$. 
Thus, we conclude that $\DR(\text{RPG}_m) = m$.

\subsection{Ordered Random Partite Graphs and Hypergraphs}\label{ss:ORPG}

Fix $m \geq 2$, and let $L = \{\edge, P_0, \ldots, P_{m-1}, <\}$ where $\edge$ is an $m$-ary relation symbol, each $P_i$ is a unary predicate symbol, and $<$ is a binary relation symbol. 
Let $\text{ORPG}_m$ denote the theory of the \emph{ordered random $m$-partite hypergraph} with hyperedge relation $\edge$, colors $P_0, \ldots, P_{m-1}$, and linear order $<$.
This structure is the Fra\"iss\'e limit of the class of all finite $L$-structures satisfying axioms (1), (3), and (4), below. Using methods similar to those used for $\RG_m$, above, we may conclude that $\ORPG_m$ has quantifier elimination and can be axiomatized by the following schema:

\begin{enumerate}
    \item The domain, or \emph{vertex set} $\mathbb{V}$, is linearly ordered by $<$.
    \item Each $P_i$, often called a \emph{color}, contains infinitely many vertices, with no least or greatest element in terms of the ordering $<$.
    \item The colors partition the vertex set so that $P_0 < \cdots < P_{m-1}$.
    \item The hyperedge relation $\edge$ is a subset of $P_0 \times \cdots \times P_{m-1}$.
    \item Given $\ell < m$ and $s,t<\omega$, if $\bar{a}_0, \dots, \bar{a}_s$, $\bar{b}_0, \dots, \bar{b}_t$ are distinct tuples in $\prod_{k\neq \ell} P_k$ and $d_0, d_1 \in P_\ell$ with $d_0 < d_1$, then there is a vertex $d \in P_\ell$ such that $d_0 < d < d_1$ and
    \[\bigwedge_{i\leq s} \bar{a}_i \edge d\ \land\ \bigwedge_{j\leq t} \bar{b}_j \noedge d.\]
\end{enumerate}
Even though we have added an ordering to the language, the exact same argument used to show that $\DR(\text{RPG}_m) = m$, above, applies to $\ORPG_m$ as well.

In \cite[Subsection 2.4]{simon:distalandnondistal}, Pierre Simon proves that if $T$ is an NIP theory, then $T$ has distality rank 1 if and only if every complete $\EM$-type whose variables are singletons has distality rank 1 (i.e., $\forall \, \Gamma\in S^{\EM} (1\cdot \omega)\; \DR(\Gamma)=1$). 
This is not true in general. 
In fact, $\ORPG_m$ is a counterexample.
In this theory, if $\II = \II_0 + \cdots + \II_{m+1} \subseteq U^1$ is an indiscernible skeleton, it must either either be constant or strictly monotonic (increasing or decreasing) and monochromatic.
In either case, it is 1-distal, but as we showed above, $\ORPG_m$ is not.

The reader may have noticed that the distality rank of every theory discussed in this section agrees with its strong distality rank.
It is unclear whether or not this agreement, at the global level, holds for all theories; however, there are cases where distality rank and strong distality rank disagree, locally, for a particular $\EM$-type.
Again, $\ORPG_m$ provides us with an example.
Let $\II = (a_r : r \in \mathbb{R}) \subseteq P_0(U)$ be an increasing indiscernible sequence, and let $\Gamma = \tp^{\EM}(\II)$. 
As we determined above, the distality rank of $\Gamma$ is 1.
Choose $b_1, \dots, b_{m-2}$ such that each $b_k \in P_k(U)$.
By (5) and compactness, there exists $b_{m-1} \in P_{m-1}(U)$ such that $a_0 \edge b_1 \cdots b_{m-1}$ but $a_r \noedge b_1 \cdots b_{m-1}$ for all $r \neq 0$, so $\Gamma$ is not strongly $(m-1)$-distal. 
In fact, we can use Corollary \ref{cor:bound SDR} to conclude that $\SDR(\Gamma) = m.$

\subsection{An Example of a Theory with Infinite Distality Rank}

Let $L = \{ \edge_2, \edge_3, \ldots \}$ where each $\edge_m$ is an $m$-ary relation, and let $\RG_\omega$ denote the theory of the Fra\"iss\'e limit of the class of all finite structures where each $\edge_m$ is a hyperedge relation (i.e., reflexive and symmetric). 
Using methods similar to those used for $\RG_m$, above, we may conclude that $\RG_\omega$ has quantifier elimination and can be axiomatized by the following schema: 
\begin{enumerate}
    \item The domain, or \emph{vertex set} $\mathbb{V}$, is infinite.
    \item For each $m \geq 2$, the relation $\edge_m$ is irreflexive and symmetric.
    \item For all integers $r \geq 2$, if $(s_2, \ldots, s_r) \subseteq \omega$ and $ (t_2, \ldots, t_r) \subseteq \omega$, then
\begin{align*}
	\forall  A_{2,0}, & \ldots, A_{2, s_2}, B_{2,0}, \ldots, B_{2,t_2} \text{ distinct } \in [\mathbb{V}]^{2-1} \\
	& \vdots \\
	\forall A_{r,0}, & \ldots, A_{r, s_r}, B_{r,0}, \ldots, B_{r,t_r} \text{ distinct } \in [\mathbb{V}]^{r-1} \\
	\\
	& \exists \, d\in \mathbb{V} \quad \left[ \bigwedge_{i\leq s_2} A_{2,i} \edge_2 d \wedge \bigwedge_{j\leq t_2} B_{2,j} \noedge_2 d \wedge \cdots \wedge \bigwedge_{i\leq s_r} A_{r,i} \edge_r d \wedge \bigwedge_{j\leq t_r} B_{r,j} \noedge_r d \right].
\end{align*}
\end{enumerate}
Furthermore, for any $m\geq 2$, by compactness, there is an indiscernible $(m-1)$-skeleton $\II=\II_0 + \cdots + \II_m \subseteq U^1$ along with elements $a_0,\ldots,a_{m-1} \in U^1$ such that the only hyperedge with vertices among $\II \cup \{a_0, \ldots, a_{m-1}\}$ is $\bar a$. Thus, we conclude that $\DR(\text{RG}_\omega) = \omega$.

\subsection{Stable Examples}\label{ss:stable examples}
According to Proposition \ref{prop:stablenotdistal}, there are no stable theories with distality rank 1. 
Since the concept of distality was originally used to decompose NIP theories into their stable and distal components, which in some sense are polar opposites, it  might seem natural to assume that distality rank separates NIP theories into a spectrum with distal theories having rank 1 and stable theories, rank $\omega$. 
This notion, however, is erroneous. 
In fact, Corollary \ref{cor:bound SDR} allows us to quickly see that several well-known stable theories have distality rank 2. 
We list a few in the next paragraph.

Let $L = \{R\}$, where $R$ is a binary relation, and fix $k > 0$.
The theory asserting that $R$ is an equivalence relation with infinitely many equivalence classes, all of which have size $k$, has distality rank 2. Furthermore, the distality rank of this theory does not change if we require all of the classes to be infinite. These examples use a relational language, but it is also easy to find examples where the language includes a function symbol. Consider the theories of $(\mathbb{N},\sigma, 0)$ and $(\mathbb{Z},\sigma)$ where $\sigma$ is the successor function. Both are stable with distality rank 2.

Another well-known stable theory, that of algebraically closed fields (ACF) in the language of rings $\{+, -, \cdot, 0, 1\}$, has infinite distality rank. Given $m > 0$, let $\II = \II_0 + \cdots + \II_{m+1} \subseteq{U^1}$ be a nonconstant indiscernible skeleton. 
Choose $a_0, \ldots, a_{m-1}\in U^1$ algebraically independent over $\II$, and let 
\[a_m = a_0 + \cdots + a_{m-1}.\]
We can insert any $m$ of the $a_i$'s, but we cannot  insert all of them. Thus, we conclude that $\DR(\ACF) > m$. Furthermore, a similar argument shows that if $T$ is any strongly minimal expansion of the theory of an infinite group, then $\DR(T) = \omega.$

Notice that we have only given examples of stable theories with distality ranks $2$ and $\omega$. In Section \ref{s:geometric stability}, we will show that these are the only ranks possible for superstable theories. 

\begin{ques} \label{ques:gap for stable}
 Can a stable theory have distality rank $m$ where $2<m<\omega$?
\end{ques}

\noindent It turns out that Question \ref{ques:gap for stable} is equivalent to a long-standing open question concerning $k$-triviality posed by Goode in \cite{goode:trivialconsiderations}. We will discuss this in more detail in Section \ref{s:conclusion}. 

\subsection{Unstable NIP Examples}
There are many unstable NIP theories with distality rank 1. In fact, using Proposition \ref{prop:stablenotdistal} together with Corollary \ref{cor:distal implies NIP} from the next section, we see that every distal theory is both unstable and NIP. 

For an example of an unstable NIP theory with distality rank 2, we can simply add a linear order to an example from the previous subsection. Consider the theory of an equivalence relation $R$ with infinitely many classes and a linear order $<$ with no endpoints such that each equivalence class is a dense subset of the domain. This is the Fra\"iss\'e limit of the class of all finite structures with an equivalence relation $R$ and a linear order $<$. 

Similarly, we can expand ACF to produce an example of an unstable NIP theory with infinite distality rank. Consider the theory of algebraically closed valued fields (ACVF) in the standard three-sorted presentation with sorts for the value group $\Gamma$ and the residue field $k$, in addition to the home sort for the valued field $K$. See Chapter 4 of \cite{marker:acvf} for more details. Given any $(K,\Gamma,k) \models \ACVF$, it follows that $k\models \ACF$. Furthermore, any subset of $k^n$ which is definable in $(K,\Gamma,k)$ is also definable in the reduct $k$ (see \cite[Corollary 4.25(i)]{marker:acvf}). Thus, we can use the same argument that we used above for ACF to conclude that $\DR(\ACVF) = \omega$.

Note that we have only given examples of NIP theories with distality ranks 1, 2, and $\omega$. 

\begin{ques} \label{ques:gap for NIP}
Can a NIP theory have distality rank $m$ where $2<m<\omega$?
\end{ques}

\noindent The answer to this question most likely depends on the answer to Question \ref{ques:gap for stable}. We will discuss this in more detail in Section \ref{s:conclusion}.

\section{Distality Rank and Shelah's Dependence Rank}\label{s:m-dependence}

Saharon Shelah introduced the notion of $m$-dependence in \cite[Section 5(H)]{shelah:stronglydependenttheories} and \cite[Definition 2.4]{shelah:definablegroupsfordependent}.
When applied to theories, this notion generalizes NIP in much the same way that $m$-distality generalizes distality. 
In particular, a theory is 1-dependent if and only if it is NIP. 
Furthermore, if a theory is $m$-dependent for some $m>0$, then it is also $n$-dependent for every $n>m$. 
After reading an earlier draft of this paper, Artem Chernikov noticed that if a theory is $m$-distal for some $m>0$, then it is also $m$-dependent. 
His result is formalized in Proposition \ref{prop:m-distal implies m-dependent}, below.

\begin{definition}\label{def:IP_m formula}
    Given $m>0$, we say a formula $\phi(x_0, \ldots, x_{m-1}, y) \in L_U$ is \emph{$m$-independent}, or IP$_m$, if we can find an infinite set $A_i \subseteq U^{|x_i|}$ for each $i<m$ such that for every subset $B \subseteq A_0 \times \cdots \times A_{m-1}$, there exists $b \in U^{|y|}$ with 
    \[
    \phi(A_0, \ldots, A_{m-1}, b) = B.
    \]
    Otherwise, we say $\phi$ is \emph{$m$-dependent}, or NIP$_m$.
\end{definition}

\noindent In light of Corollary \ref{cor:reordering variables IP_m}, below, this definition is equivalent to \cite[Definition 2.1]{chernikov:n-dependence}.

\begin{fact}
A formula $\phi \in L_U(x,y)$ is IP$_1$ if and only if it is IP.
\end{fact}

\noindent For a proof of this fact, see \cite[Lemma 2.7]{simon:guidenip}.

\begin{definition}
    Given $m>0$, we say $T$ is \emph{$m$-independent}, or IP$_m$, if some formula $\phi(x_0,\ldots,x_{m-1},y)\in L_U$ is IP$_m$. Otherwise, we say $T$ is \emph{$m$-dependent}, or NIP$_m$.
\end{definition}

\noindent It is easy to see that this property is invariant under base change. In particular, if $T$ is IP$_m$, we can always find a formula with no parameters witnessing this.

In order to prove Proposition \ref{prop:m-distal implies m-dependent}, it will be helpful to use a characterization of $m$-dependence which involves the ability to ``embed'' certain random hypergraphs.  In the following lemma, we show that for $m > 0$, a formula is IP$_m$ if and only if it interprets the hyperedge relation of a random $m$-partite hypergraph. (See also \cite[Proposition 5.2]{chernikov:n-dependence}.)

\begin{lemma}\label{lem:IP_m formula}
    Given $m>0$, a formula $\phi(x_0, \ldots, x_m) \in L_U$ is IP$_m$ if and only if there is a hypergraph
    \[
    \mathcal{G} = (\mathbb{V}, \edge, P_0, \ldots, P_m) \models \RPG_{m+1}
    \]
    and a map $f: \mathbb{V}^1 \rightarrow U^{<\omega}$, where each $f(P_k(\mathbb{V})) \subseteq U^{|x_k|}$, such that $\phi$ \underline{interprets} the hyperedge relation $\edge$, i.e., given $a_k \in P_k(\mathbb{V})$ for each $k \leq m$, we have 
    \[
    \mathcal{U} \models \phi(f(a_0),\ldots,f(a_m)) \quad \iff \quad \mathcal{G} \models \edge \bar{a}.
    \]
\end{lemma}
\begin{proof}
    ($\Rightarrow$): Suppose $\phi(x_0,\ldots,x_m)$ is IP$_m$. Choose an infinite $A_i \subseteq U^{|x_i|}$ for each $i<m$ such that for every subset $B \subseteq A_0 \times \cdots \times A_{m-1}$, there exists $b \in U^{|x_m|}$ with 
    \[
    \phi(A_0, \ldots, A_{m-1}, b) = B.
    \]
    Let $\mathcal{G} = (\mathbb{V}, \edge, P_0, \ldots, P_m)$ be a countable model of $\RPG_{m+1}$, and for each $k<m$, let $f_k:P_k(\mathbb{V})\rightarrow A_k$ be an injection. Now we can define $f_m:P_m(\mathbb{V})\rightarrow U^{|x_m|}$ so that if $a_k \in P_k(\mathbb{V})$ for each $k \leq m$, then 
    \[
    \mathcal{U} \models \phi(f_0(a_0),\ldots,f_m(a_m)) \quad \iff \quad \mathcal{G} \models \edge \bar{a}.
    \]
    ($\Leftarrow$): Let $A_i = f(P_i(\mathbb{V}))$ for each $i < m,$ and let $B \subseteq A_0\times \cdots \times A_{m-1}$. If 
    \[
    \bar{a}_0, \dots, \bar{a}_s, \bar{a}'_0, \dots, \bar{a}'_t \in A_0\times \cdots \times A_{m-1}
    \]
    with each $\bar{a}_i \in B$ and each $\bar{a}'_j\notin B$, then since $\mathcal{G}\models \RPG_{m+1}$ and $\phi$ interprets $\edge$, we can find $d \in P_m(\mathbb{V})$ such that
    \[
        \bigwedge_{i\leq s} \phi(\bar{a}_i,f(d)) \quad \land \quad \bigwedge_{j\leq t} \lnot\phi(\bar{a}'_j, f(d))
    \]
    Thus, by compactness, there exists $b \in U$ such that 
    $\phi(A_0, \ldots, A_{m-1}, b\,) = B.$
\end{proof}

\begin{cor}\label{cor:reordering variables IP_m}
    Reordering the variables does not change whether or not a formula is IP$_m$ for any $m>0$.
\end{cor}

Furthermore, as stated in the following fact, a formula is IP$_m$ if and only if it interprets the hyperedge relation of an ordered random $m$-partite graph $\mathcal{G}$ on some $\mathcal{G}$-indiscernible sequence. For a proof of this, see \cite[Proposition 5.2]{chernikov:n-dependence}. 

\begin{fact}\label{fact:IP_m generalized indiscernible witness}
    Given $m>0$, a formula $\phi(x_0, \ldots, x_m) \in L_U$ is IP$_m$ if and only if there is a hypergraph 
    \[
    \mathcal{G}=(\mathbb{V},\edge,P_0,\ldots, P_m,<)\models \ORPG_{m+1}
    \] 
    and a map $f: \mathbb{V}^1 \rightarrow U^{<\omega}$ with each $f(P_k(\mathbb{V})) \subseteq U^{|x_k|}$ such that the following hold:
    \begin{enumerate}[(i)]
        \item If two finite tuples of vertices $(a_0,\ldots,a_{n-1})$ and $(b_0,\ldots,b_{n-1})$ have the same type in $\mathcal{G}$, then their images  $(f(a_0),\ldots,f(a_{n-1}))$ and $(f(b_0),\ldots,f(b_{n-1}))$ have the same type in $\mathcal{U}$.
        \item The formula $\phi$ interprets the hyperedge relation $\edge$.
    \end{enumerate}
\end{fact}

\begin{prop}[Chernikov]\label{prop:m-distal implies m-dependent}
    Given $m>0$, if $T$ is $m$-distal, then it is also $m$-dependent.
\end{prop}

\begin{proof}
    Suppose $\phi(x_0, \ldots, x_m) \in L$ is IP$_m$.
    Let $\mathcal{G}$ and $f$ witness this as in Fact \ref{fact:IP_m generalized indiscernible witness}.
    By compactness, there is an indiscernible skeleton 
    \[
    \II = \II_0 + \cdots + \II_{m+1} \subseteq \prod_{k \leq m} P_k(\mathbb{V})
    \]
    along with tuples 
    \[
    \bar{a}_0, \ldots, \bar{a}_{m} \in \prod_{k \leq m} P_k(\mathbb{V})
    \] 
    such that the only hyperedge with vertices among $\II \bar{a}_0 \cdots \bar{a}_{m}$ is 
    \[
    (\pi_0(\bar{a}_0), \ldots, \pi_{m-1}(\bar{a}_{m}))
    \]
    where each $\pi_k : \mathbb{V}^{m+1} \to \mathbb{V}^1$ is the standard projection $(b_0, \ldots, b_{m}) \mapsto b_k$. 
    
    Given $b_k \in P_k(\mathbb{V})$ for each $k \leq m$, let $F(\bar b) = (f(b_0), \ldots, f(b_m))$.
    It follows that the skeleton
    \[
    F(\II) = F(\II_0) + \cdots + F(\II_{m+1}) \subseteq U
    \]
    is not $m$-distal since any proper subset of
    \[
    A = (F(\bar a_0), \ldots, F(\bar a_m))
    \]
    inserts indiscernibly but $A$ itself does not.
\end{proof}

\begin{cor} \label{cor:distal implies NIP}
    If $\DR(T)=1$, then $T$ is NIP.
\end{cor}

In other words, if $T$ is distal according to the original definition \cite[Definition 2.1]{simon:distalandnondistal}, then $T$ is also NIP. Although this result may be considered folklore, we believe this is the first time it has appeared in the literature. Hieronymi and Nell are credited with the proof of \cite[Proposition 2.8]{gehret:distalitytransseries} which implies that $T$ is NIP if all its indiscernible sequences satisfy the ``external characterization'' of distality (i.e., $\SDR(T)=1$). However, Corollary \ref{cor:distal implies NIP} is stronger since \cite[Lemma 2.7]{simon:distalandnondistal}, which shows that the ``external characterization'' of distality is equivalent to the original definition, assumes that the theory under consideration is NIP. 

\begin{cor}
    The following are equivalent: 
    \begin{enumerate}[(i)]
        \item  $\DR(T) = 1$.
        \item  $\SDR(T)=1$.
    \end{enumerate}
\end{cor}

\section{Type Determinacy}\label{s:type determinacy}

Let $A \subseteq U$ be a small set of parameters.

\begin{definition}\label{def:mdetermined}
	Let $n > m > 0$. Given $n$ variables $x_0, \ldots, x_{n-1}$ and a type $p\in S_A(x_0, \ldots, x_{n-1})$, we say that $p$ is \emph{$m$-determined} if it is completely determined by the types
	\[
		\left\{ q \in S_A(x_{i_0}, \ldots, x_{i_{m-1}}) \, : \, i_0 < \cdots < i_{m-1} < n \text{ and } q \subseteq p \right\}.
	\] 
\end{definition}

\begin{prop}
A Dedekind partition $\II_0 + \cdots + \II_n$ of an indiscernible sequence $\II$ is $m$-distal if and only if $\limtp_\II(\cc_0, \ldots, \cc_{n-1})$ is $m$-determined.
\end{prop}
\begin{proof}
If $a_0,\dots,a_{n-1}\in U$ and $i_0<\cdots <i_{t-1} <n $ for some $t \leq n$, then $(a_{i_0}, \ldots, a_{i_{t-1}})$ inserts into $\II_0 + \cdots + \II_n$ if and only if \[ a_{i_0},\ldots,a_{i_{t-1}}\models \limtp_\II(\mathfrak{c}_{i_0},\ldots,\mathfrak{c}_{i_{t-1}}).\]
\end{proof}

Let $\MM \models T$ be $|A|^+$-saturated with $A \subseteq M$. 

\begin{lemma}\label{lem:m-determinacy for saturated models}
	 Let $n>m>0$, and let $p\in S_U(x_0,\ldots,x_{n-1})$ be invariant over $A$. The global type $p$ is $m$-determined if and only if its restriction $p\dhr_M$ is also $m$-determined.
\end{lemma}
\begin{proof}
($\Rightarrow$): Suppose $p$ is $m$-determined. Let $\phi(x_0,\ldots,x_{n-1},y)\in  L$ and $b\in M$ be such that $\phi(\bar{x},b)\in p\dhr_M$. By compactness, for each increasing $\sigma\colon m\rightarrow n$, there is a formula $\psi_\sigma(x_{\sigma(0)},\ldots, x_{\sigma(m-1)},d)\in p$ such that 
\[
\bigwedge_{\sigma}\psi_\sigma(x_{\sigma(0)},\ldots, x_{\sigma(m-1)},d)\vdash \phi(\bar{x},b).
\]
By saturation, there is a $d'\in M$ such that $d'\equiv_{Ab}d$. It follows that 
\[
\bigwedge_{\sigma}\psi_\sigma(x_{\sigma(0)},\ldots, x_{\sigma(m-1)},d')\vdash \phi(\bar{x},b).
\]
Furthermore, for each $\sigma$, we have $\psi_\sigma(x_{\sigma(0)},\ldots, x_{\sigma(m-1)},d')\in p\dhr_M$ since $p$ is invariant over $A$. Thus, the restriction $p\dhr_M$ is $m$-determined.

($\Leftarrow$): Suppose $p\dhr_M$ is $m$-determined. Let $\phi(x,b)\in p$, and let $b'\in M$ such that $b'\equiv_A b$. By invariance, the formula $\phi(x,b')\in p$. By compactness, for each increasing $\sigma\colon m\rightarrow n$, there is a formula $\psi_\sigma(x_{\sigma(0)},\ldots, x_{\sigma(m-1)},d')\in p\dhr_M$ such that 
\[
	\bigwedge_{\sigma}\psi_\sigma(x_{\sigma(0)},\ldots, x_{\sigma(m-1)},d')\vdash \phi(x,b').
\]

\noindent Let $d\in U$ such that $bd\equiv_A b'd'$. It follows that 

\[
\bigwedge_{\sigma}\psi_\sigma(x_{\sigma(0)},\ldots, x_{\sigma(m-1)},d)\vdash \phi(x,b).
\]
Furthermore, for each $\sigma$, we have $\psi_\sigma(x_{\sigma(0)},\ldots, x_{\sigma(m-1)},d)\in p$ since $p$ is invariant over $A$. Thus, the global type $p$ is $m$-determined.
\end{proof}

Let $B\subseteq U$ be a small set of parameters, and let $\lambda$ and $\kappa$ be small cardinals.

\begin{lemma}\label{lem:collate}
	For each $\alpha < \lambda $, assume we have the following:
	\begin{itemize}
		\item a sequence of tuples $\II_\alpha$ indexed by a linear order $(I_\alpha,<)$ with a Dedekind cut $\cc_\alpha$,
		\item an initial segment $I^-_\alpha\subseteq \cc^-_\alpha$ and an end segment $I^+_\alpha\subseteq \cc^+_\alpha$, both proper,
		\item linear orders $(J^-_\alpha,<)$ and $(J^+_\alpha,<)$, and
		\item an index $$J_\alpha=I^-_\alpha + J^-_\alpha + J^+_\alpha + I^+_\alpha$$ with distinguished cut $$\dd_\alpha =(I^-_\alpha + J^-_\alpha , J^+_\alpha + I^+_\alpha).$$
	\end{itemize}
Let $(A_\beta : \beta < \kappa)$ be a family of sequences with each $A_\beta = (a^\beta _\alpha : \alpha < \lambda )\subseteq U$.
For each $\alpha < \lambda$, there is a sequence $\JJ_\alpha$ indexed by $J_\alpha$ agreeing with $\II_\alpha$ on $I^-_\alpha$ and $I^+_\alpha$ such that for all $\beta < \kappa$ and all $A'_\beta \subseteq A_\beta$, if the family $(\II_\alpha: \alpha < \lambda)$ is mutually indiscernible over $B$ after inserting each $a^\beta_\alpha\in A'_\beta$ at the corresponding cut $\cc_\alpha$, then the family $(\JJ_\alpha: \alpha < \lambda)$ is mutually indiscernible over $B$ after inserting each $a^\beta_\alpha\in A'_\beta$ at the corresponding cut $\dd_\alpha$.
\end{lemma}

\begin{proof}
	First we show that we can replace $\cc^-_0$ with $\dd_0^-$; i.e., we can find a suitable $\JJ^-_0$ so that we can replace $(b_i^0: i\in \cc_0^-)$ with $\II^-_0 + \JJ^-_0$.
	Let $I^-_0$ and $ J^-_0$ be as above.
	If $ J^-_0$ is finite, we may let $\JJ^-_0$ be any increasing sequence of the same size from $(b^0_i: I\in\cc^-_0 \setminus I^-_0)$.
	By compactness, this argument extends to the case where $J^-_0$ is infinite.
	We may now iterate to replace finitely many $\cc^\bullet_\alpha$.
	Finally, the case where $\lambda$ is infinite follows by compactness.
\end{proof}

\begin{prop}\label{prop:mutorth}
	Suppose $T$ is $m$-distal for some $m>0$. If $\II_0, \ldots, \II_{n-1}$ are mutually indiscernible over $B$, each containing a Dedekind cut $\cc_0, \ldots, \cc_{n-1}$, respectively, then $\limtp_{B\II_0\cdots \II_{n-1}}(\cc_0, \ldots, \cc_{n-1})$ is $m$-determined.
\end{prop}

\begin{proof}
	Suppose each $\II_i=(b_i^j:j\in I_i)$, and let $D=B\cup \bigcup_i \II_i$.
	Let $\hat{a}_0, \ldots, \hat{a}_{n-1} \in U$ such that 
	\begin{equation}\label{eq:mutorth1}\tag{$*$}
	\text{for all } i_0 < \cdots < i_{m-1} < n \text{ we have } \hat{a}_{i_0}\cdots \hat{a}_{i_{m-1}}\models \limtp_D(\cc_{i_0}, \ldots , \cc_{i_{m-1}}).
	\end{equation}
	Fix
	\begin{equation}\label{eq:mutorth2}\tag{$**$}
		\phi(x_0,\ldots, x_{n-1})\in \limtp_D(\cc_0, \ldots, \cc_{n-1}),
	\end{equation}
	and let $D'$ be the parameters of $\phi$. We will show that $\hat{a}_0\cdots \hat{a}_{n-1} \models \phi$.
	
	Construct an $n$-skeleton 
	\[
		\KK = \left( \left(b_0^{\sigma_0(k)}, \ldots, b_{n-1}^{\sigma_{n-1}(k)}\right) \,: \, k \in K \right)
	\] 
	with underlying index $K=K_0+ \cdots +K_n$ as follows.
	For $i<n$, let $\sigma_i: K \to I_i$ be an increasing map such that 
	\begin{itemize}
		\item $\sigma_i(K_0+ \cdots +K_i)\subseteq \cc_i^-$,
		\item $\sigma_i(K_{i+1}+ \cdots +K_n) \subseteq \cc^+_i$, and
		\item if $b_i^j\in D'$, then $j\in \sigma_i(K)$. 
	\end{itemize}
	If necessary, we can apply Lemma \ref{lem:collate} to replace a neighborhood of $\cc_i$ with one large enough to accommodate the image of $\sigma_i$ without disturbing $\II_i\cap D'$ or the validity of (\ref{eq:mutorth1}) and (\ref{eq:mutorth2}).
	
	By compactness, there is a sequence $A=(\bar{a}_0, \ldots, \bar{a}_{n-1})$ such that 
	each $\bar{a}_j=(a_0^j, \ldots, a_{n-1}^j)$ with $a^j_j=\hat{a}_j$, and
	every $m$-sized subsequence of $A$ inserts into $\KK$ indiscernibly over $B$. 
	Proposition \ref{prop:DRA} asserts that $T_B$ is $m$-distal, so $\KK$ is $m$-distal in $T_B$. It follows that the entire sequence $A$ inserts into $\KK$ indiscernibly over $B$.
	
	Since $\phi\in \limtp_D(\cc_0,\ldots, \cc_{n-1})$, if for each $j<n$, $K'_j$ is an end segment of $K_j$ such that for each $i<n$, the image $\sigma_i(K_j')$ avoids $D'$, then for all $(k_0, \ldots, k_{n-1})\in K'_0\times \cdots \times K'_{n-1}$, we have 
\[
	\UU \models \phi\left(b^{\sigma_0(k_0)}_0, \ldots, b^{\sigma_{n-1}(k_{n-1})}_{n-1}\right).
\]
	Furthermore, since $\KK\cup A$ is indiscernible over $B$, it follows that 
\[
	\UU\models \phi(\hat{a}_0,\ldots, \hat{a}_{n-1}).
\] 
\end{proof}

\begin{prop}\label{prop:distalortho}
	Suppose $T$ is $m$-distal and $n>m > 0$. If $p_0(x_0), \ldots, p_{n-1}(x_{n-1}) \in S_U$ are invariant over $B$ and commute pairwise, then the product $p_0 \otimes \cdots \otimes p_{n-1}$ is $m$-determined.
\end{prop}

\begin{proof}
	Let $p = p_0 \otimes \cdots \otimes p_{n-1}$, and let $\phi \in p$. Assume $B$ contains the parameters of $\phi$. Let $J = J_0 + \cdots + J_n$ with each $J_j = \mathbb{Z}$ in the standard order. Let $\II$ be a Morley sequence for $p$ over $B$ indexed by $J$. Let $\hat{A} = (\hat{a}_0, \ldots, \hat{a}_{n-1})$ be such that for all $\sigma: m \rightarrow n$ increasing, we have
	\[
		\hat{a}_{\sigma(0)} \cdots \hat{a}_{\sigma(m-1)} \models \left[p_{\sigma(0)} \otimes \cdots \otimes p_{\sigma(m-1)} \right] \dhr_{B \II}.
	\]
	Let $\JJ$ be a Morley sequence for $p$ over $B \cup \II \cup \hat{A}$ also indexed by $J$. For every $i < n$, let 
	\[
		\KK_i = \pi_i \left( \II_0 + \cdots + \II_i + \JJ_{i+1} + \cdots + \JJ_n \right)
	\]
	where $\pi_i$ is selecting the $i^\text{th}$ element of each tuple in the sequence,	and let
	\[
		\KK = \left( \left(b^j_0, \ldots, b^j_{n-1} \right) \, : \, j \in J  \right)
	\]
	where $b_i^j$ is the $j^\text{th}$ element of $\KK_i$. Since the $p_i$'s commute pairwise, it follows that $(\KK_i \, : \, i < n)$ is mutually indiscernible over $B$. Furthermore, the family remains mutually indiscernible over $B$ after inserting any $m$-element subset of $\hat{A}$, each $\hat{a}_i$ into $\KK_i$ at $\cc_i$. By compactness, we can find $A = (\bar{a}_0, \ldots, \bar{a}_{n-1})$ such that each $\bar{a}_j=\left(a^j_0, \ldots , a^j_{n-1}\right)$ with $a^j_j=\hat{a}_j$ and any $m$-element subset $A' \subseteq A$ inserts into $\KK$ indiscernibly over $B$. Since $\KK$ is $m$-distal, it follows that $A$ inserts indiscernibly over $B$, so $\UU \models \phi(\hat{a}_0, \ldots, \hat{a}_{n-1})$.
\end{proof}

\begin{theorem}\label{thm:distalortho}
	If $T$ is NIP and $m>0$, then the following are equivalent:
	\begin{enumerate}[(i)]
		\item $T$ is $m$-distal.
		\item For all $n> m$ and all invariant types $p_0(x_0),\ldots,p_{n-1}(x_{n-1})\in S_U$ which commute pairwise, the product $p_0\otimes \cdots \otimes p_{n-1}$ is $m$-determined.
		\item For all invariant types $p_0(x_0),\ldots,p_m(x_m)\in S_U$ which commute pairwise, the product $p_0\otimes \cdots \otimes p_m$ is $m$-determined.
	\end{enumerate}
\end{theorem}
\begin{proof}
	$(i) \Rightarrow (ii)$: Proposition \ref{prop:distalortho}.
	
	$(ii) \Rightarrow (iii)$: Immediate.
	
	$(iii) \Rightarrow (i)$: Assume $(iii)$ holds but $(i)$ does not. Let the skeleton $\II = \II_0+\cdots +\II_{m+1}$ and $(\phi, A, B)$ witness that $T$ is not $m$-distal (see Definition \ref{def:witness}).
	Fact \ref{fact:finsat} asserts that each $\lim(\cc_i^-)$ is invariant over $\II$.
	Furthermore, Lemma \ref{lem:limmutind} asserts that
	\[
	\lim(\cc_0^-,\ldots, \cc_m^-) = \lim(\cc_0^-)\otimes \cdots \otimes \lim(\cc_m^-)
	\]
	and that $\lim(\cc_i^-)$ commutes with $\lim(\cc_j^-)$ for $i\neq j$, so the product $\lim(\cc_0^-,\ldots, \cc_m^-)$ is $m$-determined.
	By compactness, we can choose 
	 \[
	 \psi_\sigma\left(x_{\sigma(0)},\ldots, x_{\sigma(m-1)}\right)\in \lim\left(\cc_{\sigma(0)}^-, \ldots, \cc_{\sigma(m-1)}^-\right)
	 \]
	 for each increasing map $\sigma:m\to m+1$ such that 
	 \begin{equation} \label{eq:mdet} \tag{$*$}
	 \bigwedge_{\sigma} \psi_\sigma\left(x_{\sigma(0)}, \ldots, x_{\sigma(m-1)}\right) \quad \vdash \quad \phi(\bar{b}_0,x_0, \ldots, \bar{b}_m,x_{m},\bar{b}_{m+1}).
	 \end{equation}
	 Let $D$ be the parameters of $\bigwedge_{\sigma} \psi_\sigma$.
	 By the Base Change Lemma (Lemma \ref{lem:strongbasechange}), there is $A' \equiv_\II A$ such that for each $\sigma$, we have
	 \[
	 a'_{\sigma(0)}\cdots a'_{\sigma(m-1)} \models \limtp_D\left(\cc_{\sigma(0)}^-,\cdots,\cc_{\sigma(m-1)}^-\right),
	 \]
	 but this contradicts (\ref{eq:mdet}) since 
	 \[
	 a'_0 \cdots a'_m \not \models \phi(\bar{b}_0,x_0, \ldots, \bar{b}_m,x_{m},\bar{b}_{m+1}).
	 \]
\end{proof}

\section{Distality Rank and Geometric Stability}\label{s:geometric stability}

Throughout this section, we assume $T$ is stable.

\begin{definition}
    Given $k>0$, we say $T$ is \emph{$k$-trivial} if over every small $D \subseteq U$, there are no $(k+2)$-cycles.
    Moreover, we say $T$ is \emph{trivial} if it is $1$-trivial.
\end{definition}

\noindent See Section 1 of \cite{goode:trivialconsiderations}.

\begin{lemma}\label{lem:cycle implies not determined}
Given $m>0$, a small model $\MM \models T$, and tuples $a_0, \dots, a_m \in U^{<\omega}$, if $\bar{a}$ is an $(m+1)$-cycle over $M$, then the product
\[\tp_M(a_0)\uhr^U \otimes \cdots \otimes \tp_M(a_{m})\uhr^U\]
is not $m$-determined.
\end{lemma}
\begin{proof}
Suppose $\bar{a}$ is an $(m+1)$-cycle over $M$. Lemma \ref{lem:independent iff Morley} asserts that
\[\bar{a}\setminus a_j \models \left(\bigotimes_{i\neq j}\tp_M(a_i)\uhr^U\right)\dhr_M\]
for each $j \leq m$.
Let $q$ be the unique global nonforking extension of $\tp_M(\bar{a})$.
By stationarity, we must have 
\[q\cap L(\bar{x}\setminus x_j) =\bigotimes_{i\neq j}\tp_M(a_i)\uhr^U\]
for each $j \leq m$. 
However, since $\bar a$ is dependent over $M$, Lemma \ref{lem:independent iff Morley} asserts that 
\[q\dhr_M \neq \left(\bigotimes_{i}\tp_M(a_i)\uhr^U\right)\dhr_M.\]
\end{proof}

\begin{prop}\label{prop:distal implies trivial}
Given $k>0$, if $T$ is $(k+1)$-distal, then it must also be $k$-trivial.
\end{prop}
\begin{proof}
Suppose $T$ is not $k$-trivial. 
Using cycle extension (Lemma \ref{lem:cycle extension}), we can find a small model $\mathcal{M} \models T$ and a $(k+2)$-cycle $\bar{a}$ over $M$. By Lemma \ref{lem:cycle implies not determined}, the product 
\[
\tp_M(a_0)\uhr^U \otimes \cdots \otimes \tp_M(a_{k+1})\uhr^U
\]
is not $(k+1)$-determined, so the result follows by Theorem \ref{thm:distalortho}.
\end{proof}

Since the product of any two global invariant types commute in the stable context (Lemma \ref{lem:stable types commute}), there is an analogue of Theorem \ref{thm:distalortho} involving powers of invariant types rather than mixed-factor products.

\begin{prop}\label{prop:stable m-determined products}
Given $m>0$, the following are equivalent:
\begin{enumerate}[(i)]
    \item $T$ is $m$-distal.
    \item For every global invariant type $p\in S_U$ and every $n > m$, the product $p^n$ is $m$-determined.
    \item For every global invariant type $p\in S_U$ and every $n > m$, the product $p^n$ is $(n-1)$-determined.
\end{enumerate} 
\end{prop}
\begin{proof}
$(i) \Rightarrow (ii)$: Theorem \ref{thm:distalortho}.

$(ii) \Rightarrow (iii)$: Obvious.

$(iii) \Rightarrow (i)$: Suppose $T$ is not $m$-distal. There exists $\Gamma \in S^{\EM}$ which is not $m$-distal. Choose a skeleton $\II_0 + \II_1 \models^{\EM} \Gamma.$ Lemma \ref{lem:Morley sequence for ultralimit} asserts that $\II_0 \models p^\omega\dhr_{\II_1}$ where $p = \lim(\II_1^\ast)$.
Let $\MM \models T$ be $\aleph_1$-saturated with $\II_1 \subseteq M$. 
Choose a skeleton 
\[
\JJ = \JJ_0 + \cdots + \JJ_{m+1} \models^{\EM} p^\omega \dhr_M.
\]
By the Base Change Lemma (Lemma \ref{lem:strongbasechange}), there exists $(\phi, A, B)$ witnessing that $\JJ$ is not $m$-distal in $T_M$. 
It follows that $p^n \dhr_M$ is not $(n-1)$-determined, where $n = |A|+|B|$, so by Lemma \ref{lem:m-determinacy for saturated models}, the global product $p^n$ is not $(n-1)$-determined.
\end{proof}

\begin{lemma}\label{lem:stable cycles}
Given $n>m>0$ and a global invariant type $p \in S_U(x)$, if $p^{n}$ is not $(n-1)$-determined, then for some small base $D \subseteq U$ over which $p$ is stationary, the set of realizations of $p \dhr_D$ contains an $(m+1)$-cycle over $D$.
\end{lemma}

\begin{proof}
Suppose $p$ is invariant over some small $B \subseteq U$ and the product $p^{n}$ is not $(n-1)$-determined. 
Choose $\phi(\bar{x},d)\in p^{n}$ such that 
\[
\bigcup_{i<n}p^{n-1}(\bar{x}\setminus x_i)\cup \{\lnot \phi(\bar{x},d)\}
\] 
is consistent. 
Let $\MM \models T$ be a small model containing $Bd$, and let
\[
\bar a \models \bigcup_{j \leq m} p^{n-1}\dhr_{M}(\bar x \setminus x_i) \cup \{\lnot\phi(\bar{x},d)\}.
\]
Lemma \ref{lem:independent iff Morley} asserts that $a_0 \cdots a_m$ is an $(m+1)$-cycle over $M a_{m+1} \cdots a_{n-1}$.
\end{proof}

\begin{prop}\label{prop:stable distal cycles}
    Given $m > 0$, the following are equivalent:
    \begin{enumerate}[(i)]
        \item $T$ is $m$-distal.
        \item For all small models $\MM \models T$ and all types $p \in S_M$, the set of realizations $p(U)$ contains no $(m+1)$-cycles over $M$.
        \item For all small sets $D \subseteq U$ and all types $p \in S_D$, if $p$ is stationary, then its set of realizations $p(U)$ contains no $(m+1)$-cycles over $D$.
    \end{enumerate}
\end{prop}

\begin{proof}
    $(i) \Rightarrow (ii)$: Lemma \ref{lem:cycle implies not determined} and Theorem \ref{thm:distalortho}.
    
    $(ii) \Rightarrow (iii)$: Cycle Extension (Lemma \ref{lem:cycle extension}).
    
    $(iii) \Rightarrow (i)$: Proposition \ref{prop:stable m-determined products} and Lemma \ref{lem:stable cycles}.
\end{proof}

\begin{theorem}\label{thm:distal iff trivial}
Given $k>0$, it follows that $T$ is $k$-trivial if and only if it is $(k+1)$-distal.
\end{theorem}
\begin{proof}
($\Rightarrow$): Proposition \ref{prop:stable distal cycles}.

($\Leftarrow$): Proposition \ref{prop:distal implies trivial}.
\end{proof}

\begin{definition}
    Given $k>0$, a small set $D \subseteq U$, and a type $p \in S_D$, we say that $p$ is \emph{$k$-trivial} if its realization set $p(U)$ contains no $(k+2)$-cycles over $D$.
\end{definition}

\begin{cor}
    If $T$ is not $k$-trivial, then for some small model $\MM \models T$, there is a type $p \in S_M$ which is not $k$-trivial.
\end{cor}

\begin{prop}\label{prop:DR invariant Teq}
    In the stable context, distality rank does not change when passing to $T^{\eq}$; i.e., $\DR(T) = \DR(T^{\eq})$.
\end{prop}

\begin{proof}
    Theorem \ref{thm:distal iff trivial} and Lemma \ref{lem:imaginary cycles}.
\end{proof}

\begin{prop}\label{prop:DR gap for superstable}
    If $T$ is superstable, then the following hold:
    \begin{enumerate}[(i)]
        \item $T$ is trivial if and only if $\DR(T) = 2$.
        \item $T$ is not trivial if and only if $\DR(T) = \omega$.
    \end{enumerate}
\end{prop}
\begin{proof}
Suppose $T$ is not trivial. After potentially passing to $T^{\eq}$, by \cite[Proposition 2]{goode:trivialconsiderations}, we can find a $3$-cycle among the set of realizations of some regular type. Since forking dependence defines a pregeometry on this set, we can ``expand'' the $3$-cycle to form arbitrarily large cycles using the construction outlined in the proof of \cite[Proposition 3]{goode:trivialconsiderations}. The result now follows by Theorem \ref{thm:distal iff trivial} and Proposition \ref{prop:DR invariant Teq}.
\end{proof}

\section{Conclusion} \label{s:conclusion}
Distality rank and strong distality rank provide new ways to classify $\EM$-types and theories. Strong distality rank is always greater than or equal to distality rank (Proposition \ref{prop:strongdistal}), and in Subsection \ref{ss:ORPG}, we gave an example of an $\EM$-type where this inequality is strict. Currently, we are not aware of any theories for which the two ranks disagree.

In Section \ref{s:examples}, we showed that both the distality rank and the strong distality rank hierarchies are full, giving examples of theories for each rank from $1$ to $\omega$. Distal theories are precisely those with distality rank 1 or, equivalently, strong distality rank 1.
IP theories preclude rank $1$ (Corollary \ref{cor:distal implies NIP}) but may have any other rank from $2$ to $\omega$. 
Stable theories also preclude rank $1$ (Proposition \ref{prop:stablenotdistal}). 
Moreover, superstable theories may only have distality rank $2$ or $\omega$, with nothing in between (Proposition \ref{prop:DR gap for superstable}). 
Whether or not this gap persists for stable theories is equivalent to a long-standing open question posed by Goode in 1991 at the conclusion of \cite{goode:trivialconsiderations}.
\begin{ques}[Goode]\label{ques:goode}
 Can a stable theory be $k$-trivial for some $k>1$ but not trivial?
\end{ques}

Distality rank gives us a new way to approach Goode's question. Furthermore, when restated in terms of distality rank, it naturally extends to a general context which is no longer restricted to stable theories. 

\begin{ques}\label{ques:gap for property}
Given a certain model theoretic tameness property, do any theories with that property have distality rank $m$ such that $2<m<\omega$? 
\end{ques}

\noindent For stable theories, this is Question \ref{ques:gap for stable} which, according to Theorem \ref{thm:distal iff trivial}, is equivalent to Goode's question. For NIP theories, this is Question \ref{ques:gap for NIP}.
If there is an unstable NIP theory in the gap, it seems likely we could produce a stable theory in the gap using decomposition, removing the distal part and keeping the stable part.
Thus, we conjecture that Questions \ref{ques:gap for stable} and \ref{ques:gap for NIP} are equivalent. Due to the strong relationship between $m$-distality and $k$-triviality, we believe that other ideas from geometric stability theory can be exported beyond the tame realm of stable theories (or even simple theories) using a dependence relation based on $m$-distality.

Since the concept of distality was originally used to decompose NIP theories into their stable and distal components, it  might seem natural to assume that distality rank separates NIP theories into a spectrum with distal theories having rank 1 and stable theories, rank $\omega$. 
This, however, is erroneous since several stable theories have rank 2.
Rather, distality rank measures the freedom with which we can combine single-variable types to form multi-variable types. 
Specifically, in Section \ref{s:type determinacy}, we see that $m$-distality requires certain products to be $m$-determined (Proposition \ref{prop:distalortho}). Furthermore, for NIP theories, this behavior fully characterizes $m$-distality (Theorem \ref{thm:distalortho}). 
Other nice properties hold in an NIP context.
For example, the order type of a Dedekind partition does not affect its distality rank (Theorem \ref{thm:order type does not matter for NIP}), and the distality rank of a theory is invariant under base change (Theorem \ref{thm:basechange}). 
It would be interesting to determine whether or not these properties continue to hold outside NIP.
Regardless of context, adding named parameters does not increase the distality rank or the strong distality rank of a theory, and the order type of a Dedekind partition does not affect its strong distality rank (Corollary \ref{cor:SDR otype}).

Finally, since $m$-distality is a strengthening of $m$-dependence (Proposition \ref{prop:m-distal implies m-dependent}), we expect further research to garner improvements to several existing combinatorial results for $m$-dependent structures in much the same fashion that research in distality has improved results previously known for NIP.  
In particular, we conjecture that requiring $m$-distality, instead of $m$-dependence, will yield a more homogeneous version of the Chernikov--Towsner hypergraph regularity lemma \cite[Theorem 1.1]{chernikov:HypergraphRegularity}.

\bibliography{bib}

\providecommand{\bysame}{\leavevmode\hbox to3em{\hrulefill}\thinspace}
\providecommand{\MR}{\relax\ifhmode\unskip\space\fi MR }
\providecommand{\MRhref}[2]{%
  \href{http://www.ams.org/mathscinet-getitem?mr=#1}{#2}
}
\providecommand{\href}[2]{#2}
\begin{thebibliography}{10}

\bibitem{aschenbrenner:distalityinvaluedfields}
Matthias Aschenbrenner, Artem Chernikov, Allen Gehret, and Martin Ziegler,
  \emph{Distality in valued fields and related structures}, Preprint,
  arXiv:2008.09889 (2020).

\bibitem{boxall:definablepq}
Gareth Boxall and Charlotte Kestner, \emph{The definable {$(p,q)$}-theorem for
  distal theories}, J. Symb. Log. \textbf{83} (2018), no.~1, 123--127.
  \MR{3796278}

\bibitem{chernikov:cuttinglemma}
Artem Chernikov, David Galvin, and Sergei Starchenko, \emph{Cutting lemma and
  {Z}arankiewicz's problem in distal structures}, Selecta Math. (N.S.)
  \textbf{26} (2020), no.~2, Paper No. 25, 27. \MR{4079189}

\bibitem{chernikov:Mekler}
Artem Chernikov and Nadja Hempel, \emph{Mekler's construction and generalized
  stability}, Israel J. Math. \textbf{230} (2019), no.~2, 745--769.
  \MR{3940434}

\bibitem{chernikov:n-depgroupfieldsII}
\bysame, \emph{On {$n$}-dependent groups and fields {II}}, Forum Math. Sigma
  \textbf{9} (2021), Paper No. e38, 51. \MR{4258515}

\bibitem{chernikov:n-dependence}
Artem Chernikov, Daniel Palacin, and Kota Takeuchi, \emph{On {$n$}-dependence},
  Notre Dame J. Form. Log. \textbf{60} (2019), no.~2, 195--214. \MR{3952231}

\bibitem{chernikov:externallydefinableII}
Artem Chernikov and Pierre Simon, \emph{Externally definable sets and dependent
  pairs {II}}, Trans. Amer. Math. Soc. \textbf{367} (2015), no.~7, 5217--5235.
  \MR{3335415}

\bibitem{chernikov:regularitylemma}
Artem Chernikov and Sergei Starchenko, \emph{Regularity lemma for distal
  structures}, J. Eur. Math. Soc. (JEMS) \textbf{20} (2018), no.~10,
  2437--2466. \MR{3852184}

\bibitem{chernikov:HypergraphRegularity}
Artem Chernikov and Henry Towsner, \emph{Hypergraph regularity and higher arity
  vc-dimension}, Preprint, arXiv:2010.00726 (2020).

\bibitem{gehret:distalitytransseries}
Allen Gehret and Elliot Kaplan, \emph{Distality for the asymptotic couple of
  the field of logarithmic transseries}, Notre Dame J. Form. Log. \textbf{61}
  (2020), no.~2, 341--361. \MR{4092539}

\bibitem{goode:trivialconsiderations}
John~B. Goode, \emph{Some trivial considerations}, J. Symbolic Logic
  \textbf{56} (1991), no.~2, 624--631. \MR{1133089}

\bibitem{hempel:ndepgroupfields}
Nadja Hempel, \emph{On {$n$}-dependent groups and fields}, MLQ Math. Log. Q.
  \textbf{62} (2016), no.~3, 215--224. \MR{3509704}

\bibitem{hieronymi:distalpairs}
Philipp Hieronymi and Travis Nell, \emph{Distal and non-distal pairs}, J. Symb.
  Log. \textbf{82} (2017), no.~1, 375--383. \MR{3631293}

\bibitem{hodges:shortermodeltheory}
Wilfrid Hodges, \emph{A shorter model theory}, Cambridge University Press,
  Cambridge, 1997. \MR{1462612}

\bibitem{kaplan:exactsaturation}
Itay Kaplan, Saharon Shelah, and Pierre Simon, \emph{Exact saturation in simple
  and {NIP} theories}, J. Math. Log. \textbf{17} (2017), no.~1, 1750001, 18.
  \MR{3651210}

\bibitem{marker:modeltheory}
David Marker, \emph{Model theory: An introduction}, Graduate Texts in
  Mathematics, vol. 217, Springer-Verlag, New York, 2002. \MR{1924282}

\bibitem{marker:acvf}
\bysame, \emph{Model theory of valued fields}, AMS Open Math Notes, 2019,
  OMN:201906.110798.

\bibitem{nell:distalbehavior}
Travis Nell, \emph{Distal and non-distal behavior in pairs}, MLQ Math. Log. Q.
  \textbf{65} (2019), no.~1, 23--36. \MR{3957384}

\bibitem{shelah:stronglydependenttheories}
Saharon Shelah, \emph{Strongly dependent theories}, Israel J. Math.
  \textbf{204} (2014), no.~1, 1--83. \MR{3273451}

\bibitem{shelah:definablegroupsfordependent}
\bysame, \emph{Definable groups for dependent and 2-dependent theories},
  Sarajevo J. Math. \textbf{13(25)} (2017), no.~1, 3--25. \MR{3666349}

\bibitem{simon:distalandnondistal}
Pierre Simon, \emph{Distal and non-distal {NIP} theories}, Ann. Pure Appl.
  Logic \textbf{164} (2013), no.~3, 294--318. \MR{3001548}

\bibitem{simon:guidenip}
\bysame, \emph{A guide to {NIP} theories}, Lecture Notes in Logic, vol.~44,
  Association for Symbolic Logic; Cambridge Scientific Publishers, Cambridge,
  2015. \MR{3560428}

\bibitem{simon:typedecomposition}
\bysame, \emph{Type decomposition in {NIP} theories}, J. Eur. Math. Soc. (JEMS)
  \textbf{22} (2020), no.~2, 455--476. \MR{4049222}

\bibitem{tent:modeltheory}
Katrin Tent and Martin Ziegler, \emph{A course in model theory}, Lecture Notes
  in Logic, vol.~40, Association for Symbolic Logic; Cambridge University
  Press, Cambridge, 2012. \MR{2908005}

\bibitem{terry:vcldimension}
C.~Terry, \emph{{$VC_{\ell}$}-dimension and the jump to the fastest speed of a
  hereditary {$\mathcal{L}$}-property}, Proc. Amer. Math. Soc. \textbf{146}
  (2018), no.~7, 3111--3126. \MR{3787371}

\bibitem{wilson:introductiontographtheory}
Robin~J. Wilson, \emph{Introduction to graph theory}, Longman, Harlow, 1996,
  Fourth edition [of MR0357175]. \MR{2590569}

\end{thebibliography}
\bibliographystyle{amsplain}

\end{document}